\documentclass[11pt]{amsart}

\usepackage{amsmath}
\usepackage{amssymb}
\usepackage{amsfonts}
\usepackage{overpic}
\usepackage{enumitem}   
\usepackage{graphicx} 
\usepackage{amsrefs}
\usepackage[a4paper,margin=2cm]{geometry}
\usepackage[hidelinks]{hyperref}

\graphicspath{{Figures/}}

\newtheorem{theorem}{Theorem}

\newtheorem{proposition}{Proposition} 
\newtheorem{corollary}{Corollary} 
\newtheorem{remark}{Remark} 
\newtheorem{lemma}{Lemma} 
\newtheorem{definition}{Definition}

\begin{document}
	
\title[Polycycles in non-smooth planar vector fields]{Stability and cyclicity of polycycles in non-smooth planar vector fields}

\author[Paulo Santana]
{Paulo Santana$^1$}

\address{$^1$ IBILCE--UNESP, CEP 15054--000, S. J. Rio Preto, S\~ao Paulo, Brazil}
\email{paulo.santana@unesp.br}

\subjclass[2020]{34A36, 34C23, 34C37, 37C29, 37G15}

\keywords{Polycycles, non-smooth vector fields, piecewise vector fields, Filippov systems.}

\begin{abstract}
	In this paper we extend three results about polycycles (also known as graphs) of planar smooth vector field to planar non-smooth vector fields (also known as piecewise vector fields, or Filippov systems). The polycycles considered here may contain hyperbolic saddles, semi-hyperbolic saddles, saddle-nodes and tangential singularities of any degree. We determine when the polycycle is stable or unstable. We prove the bifurcation of at most one limit cycle in some conditions and at least one limit cycle for each singularity in other conditions.
\end{abstract}

\maketitle

\section{Introduction and description of the results}\label{Intro}

The field of Dynamic Systems has developed and now have many branches, being one of them the field of \emph{non-smooth vector fields} (also known as piecewise vector fields, or Filippov systems), a common frontier between mathematics, physics and engineering. See \cites{Andronov,Filippov} for the pioneering works in this area. For applications, see \cites{Jef1,Jef2,Jef3} and the references therein. In this paper we are interested in the \emph{qualitative theory} of non-smooth vector fields. More precisely, in the qualitative theory of polycycles in non-smooth vector fields. A polycycle is a simple closed curve composed by a collection of singularities and regular orbits, inducing a first return map. There are many works in the literature about polycycles in smooth vector fields, take for example some works about its stability \cites{Cherkas,Dulac1923,Soto,GasManMan}, the number of limit cycles which bifurcates from it \cites{Mourtada,DumRouRou,HanWuBi,YeCaiLo,DumMorRou}, the displacement maps \cites{HolmesHeteroclinic,GuckHolmes,PerkoHeteroclinic,Duff} and some bifurcation diagrams \cites{DumRouRou,Mourtada2}. There are also some literature about polycycles in non-smooth vector fields, dealing for example with bifurcation diagrams \cites{AndGomNov,NovTeiZel,NovRon} and the Dulac problem \cite{AndJefMarTei}. 

The goal of this paper is to extend to non-smooth vector fields three results about polycycles in smooth vector fields. To do this, we lay, as in the smooth case, mainly in the idea of obtaining global properties of the polycycle from local properties of its singularities. For a brief description of the obtained results, let $Z$ be a non-smooth vector field with a polycycle $\Gamma^n$ with $n$ singularities $p_i$, each of them being either a hyperbolic saddle or a tangential singularity. For each $p_i$, we associated a positive real number $r_i$ such that if $r_i>1$ (resp. $r_i<1$), then $p_i$ locally contract (resp. repels) the flow. Our first main result deals with the stability of $\Gamma^n$, stating that if $r(\Gamma^n)=\prod_{i=1}^{n}r_i$ is such that $r(\Gamma^n)>1$ (resp. $r(\Gamma^n)<1)$, then the polycycle contract (resp. repels) the flow. Our second and third main results deals with the number of limit cycles that can bifurcate from $\Gamma^n$. More precisely, in our second main result we state sufficient conditions so that the cyclicity of $\Gamma^n$ is one and in our third main result we state sufficient conditions so that the cyclicity of $\Gamma^n$ is at least $n$.

The paper is organized as follows. In Section~\ref{MR} we establish the main theorems. In Section~\ref{sec2} we have some preliminaries about the transitions maps near a hyperbolic saddle, a semi-hyperbolic singularity with a hyperbolic sector, and a tangential singularity. Theorems~\ref{Main1} and \ref{Main3} are proved in Section~\ref{sec3}. In Sections~\ref{sec4} and \ref{sec5} we study some tools to approach Theorem~\ref{Main4}, which is proved in Section~\ref{sec6}. 

\section{Main Results}\label{MR}

Let $h_i\colon\mathbb{R}^2\to\mathbb{R}$, $i\in\{1,\dots,N\}$, $N\geqslant1$, be $C^\infty$-real functions. For these functions, define $\Sigma_i=h^{-1}(\{0\})$. Suppose also that $0$ is a regular value of $h_i$, i.e. $\nabla h_i(x)\neq0$ for every $x\in\Sigma_i$, $i\in\{1,\dots,N\}$. Define $\Sigma=\cup_{i=1}^{N}\Sigma_i$ and let $A_1,\dots, A_M$, $M\geqslant 2$, be the connected components of $\mathbb{R}^2\backslash\Sigma$. For each $j\in\{1,\dots,M\}$, let $\overline{A}_j$ be the topological closure of $A_j$ and let $X_j$ be a $C^\infty$-planar vector field defined over $\overline{A}_j$.

\begin{definition}
	Given $\Sigma$, $A_1,\dots, A_M$ and $X_1,\dots X_M$ as above, the associated planar non-smooth vector field $Z=(X_1,\dots, X_M;\Sigma)$, with discontinuity $\Sigma$, is the non-smooth planar vector field given by $Z(x,\mu)=X_j(x,\mu)$, if $x\in A_j$, for some $j\in\{1,\dots,M\}$. In this case, we say that the vector fields $X_1,\dots,X_M$ are the components of $Z$ and $\Sigma_1,\dots,\Sigma_N$ are the components of $\Sigma$.
\end{definition}

From now on, let us denote by $p$ points on $\Sigma$ such that there exists an unique $i\in\{1,\dots,N\}$ such that $p\in\Sigma_i$. Let also $X$ be one of the two components of $Z$ defined at $p$. The \emph{Lie derivative} of $h_i$ in the direction of the vector field $X$ at $p$ is defined as,
	\[Xh_i(p)=\left<X(p),\nabla h_i(p)\right>,\]
where $\left<\;,\;\right>$ denotes the standard inner product of $\mathbb{R}^2$. Under these conditions, we say that $p$ is a \emph{tangential singularity} if
		\[X_ah_i(p)X_bh_i(p)=0, \quad X_a(p)\neq0, \quad X_b(p)\neq0,\] 
where $X_a$ and $X_b$ are the two components of $Z$ defined at $p$. Let $x\in\Sigma$. We say that $x$ is a \emph{crossing point} if there exist an unique $i\in\{1,\dots,N\}$ such that $x\in\Sigma_i$ and $X_ah_i(x)X_bh_i(x)>0$, where $X_a$ and $X_b$ are the two components of $Z$ defined at $x$.

\begin{definition}
	A graphic of $Z$ is a subset formed by singularities $p_1,\dots p_n,p_{n+1}=p_1$, (not necessarily distinct) and regular orbits $L_1,\dots, L_n$ such that $L_i$ is a stable characteristic orbit of $p_i$ and an unstable characteristic orbit of $p_{i+1}$ (i.e. $\omega(L_i)=p_i$ and $\alpha(L_i)=p_{i+1}$), oriented in the sense of the flow. A polycycle is a graphic with a return map. A polycycle $\Gamma^n$ is semi-elementary if it satisfies the following conditions.
	\begin{enumerate}[label=(\alph*)]
		\item Each regular orbit $L_i$ intersects $\Sigma$ at most in a finite number of points $\{x_{i,0},x_{i,1},\dots,x_{i,n(i)}\}$, with each $x_{i,j}$ being a crossing point;
		\item $\Gamma^n$ is homeomorphic to $\mathbb{S}^1$;
		\item Each singularity $p_i$ satisfies exactly one of the following conditions:
		\begin{enumerate}[label=(\roman*)]
			\item $p_i$ is semi-hyperbolic and $p_i\not\in\Sigma$;
			\item $p_i$ is a hyperbolic saddle and $p_i\not\in\Sigma$;
			\item $p_i$ is a tangential singularity.
		\end{enumerate}
	\end{enumerate}
	A polycycle is elementary if it satisfies conditions $(a)$, $(b)$ and if its singularities satisfies either $(ii)$ or $(iii)$.
\end{definition}

From now on, let $\Gamma^n$ denote an elementary or semi-elementary polycycle with $n$ distinct singularities $p_1,\dots,p_n$. See Figure~\ref{Fig1}. 
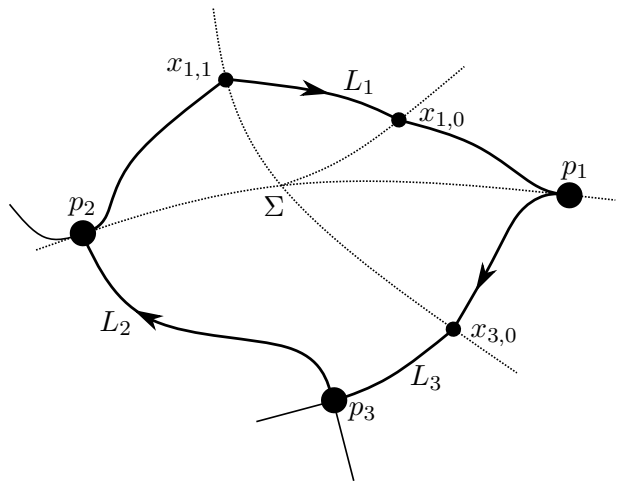
\begin{figure}[ht]
	\begin{center}
		\begin{overpic}[width=8cm]{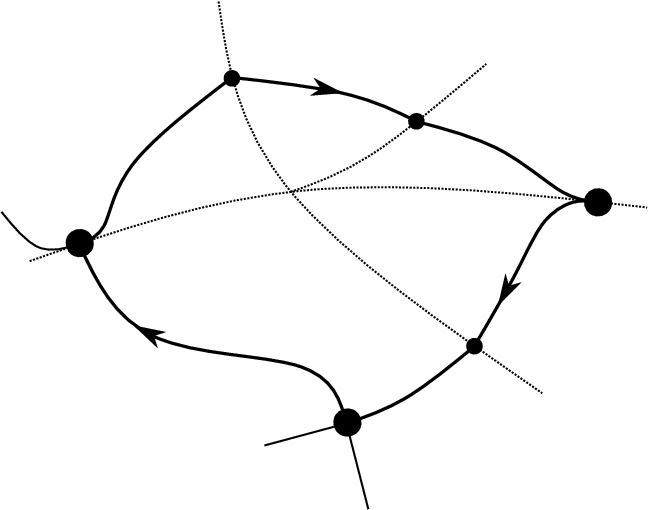} 
			\put(91,51){$p_1$}
			\put(10,45){$p_2$}
			\put(56,11){$p_3$}
			\put(55,65){$L_1$}
			\put(15,25){$L_2$}
			\put(66,16){$L_3$}
			\put(67.5,59.75){$x_{1,0}$}
			\put(26,68){$x_{1,1}$}
			\put(76,24){$x_{3,0}$}
			\put(42,44){$\Sigma$}
		\end{overpic}
	\end{center}
	\caption{An example of $\Gamma^3$. Observe that $\Gamma^3$ does not pass through the intersection of the components of $\Sigma$.}\label{Fig1}
\end{figure}
Observe that $\Gamma^n$ divide the plane in two connected sets, with only one being bounded. Let $A$ denote the connected set in which the first return map is contained. Observe that $A$ can be either the bounded or unbounded set delimited by $\Gamma^n$. See Figure~\ref{Fig8}.
\begin{figure}[ht]
	\begin{center}
		\begin{minipage}{5.5cm}
			\begin{center}
				\begin{overpic}[height=4cm]{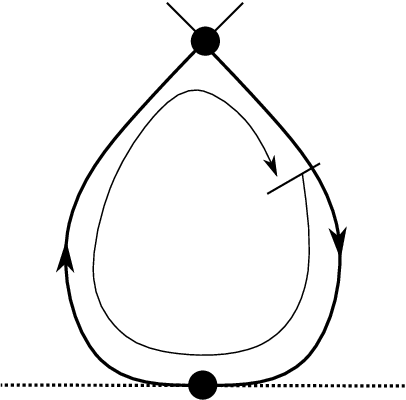} 
				\end{overpic}
				
				$(a)$
			\end{center}
		\end{minipage}
		\begin{minipage}{5.5cm}
			\begin{center}
				\begin{overpic}[height=4cm]{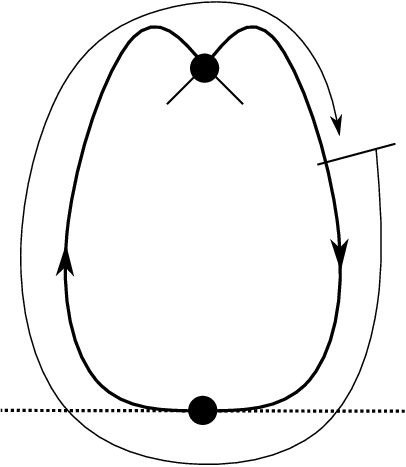} 
				\end{overpic}
				
				$(b)$
			\end{center}
		\end{minipage}
	\end{center}
	\caption{Examples of polycycles $\Gamma^2$ such that $(a)$ $A$ is the bounded set and $(b)$ $A$ is the unbounded set.}\label{Fig8}
\end{figure}

\begin{definition}
	Let $p\in\Sigma_i$ be a tangential singularity, $X$ one of the components of $Z$ defined at $p$ and let $X^kh_i(p)=\left<X(p),\nabla X^{k-1}h_i(p)\right>$, $k\geqslant2$. We say that $X$ has $m$-order contact with $\Sigma$ at $p$, $m\geqslant1$, if $m$ is the first positive integer such that  $X^mh_i(p)\neq0$. 
\end{definition}

Let $p\in\Sigma_i$ be a tangential singularity of $\Gamma^n$, $L_s$ and $L_u$ the regular orbits of $\Gamma^n$ such that $\omega(L_s)=p$ and $\alpha(L_u)=p$. Let $X_a$ and $X_b$ be the two components of $Z$ defined at $p$ and let $A_a$, $A_b$ be the respective connected components of $\mathbb{R}^2\backslash\Sigma$ such that $X_a$ and $X_b$ are defined over $\overline{A_a}$ and $\overline{A_b}$. Given two parametrizations $\gamma_s(t)$ and $\gamma_u(t)$ of $L_s$ and $L_u$ such that $\gamma_s(0)=\gamma_u(0)=p$, let $A_s$, $A_u\in\{A_a,A_b\}$ be such that $A_s\cap\gamma_s([-\varepsilon,0])\neq\emptyset$ and $A_u\cap\gamma_u([0,\varepsilon])\neq\emptyset$, for any $\varepsilon>0$ small. Let also $X_s$, $X_u\in\{X_a,X_b\}$ denote the components of $Z$ defined at $A_s$ and $A_u$. Observe that we may have $A_s=A_u$ and thus $X_s=X_u$. See Figure~\ref{Fig18}. 

\begin{figure}[ht]
	\begin{center}
		\begin{minipage}{6cm}
			\begin{center}
				\begin{overpic}[height=4cm]{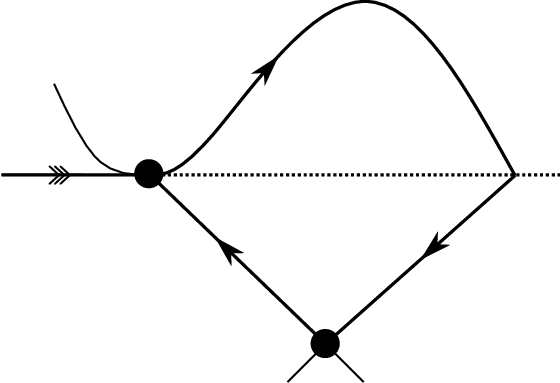} 
					\put(25,42){$p$}
					\put(83,54){$L_u$}
					\put(33,18){$L_s$}
				\end{overpic}
				
				$(a)$
			\end{center}
		\end{minipage}
		\begin{minipage}{6cm}
			\begin{center}
				\begin{overpic}[height=4cm]{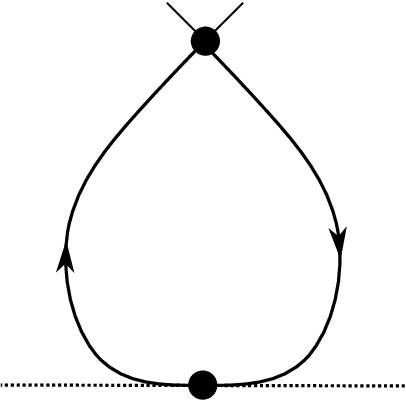} 
					\put(48,10){$p$}
					\put(6,20){$L_u$}
					\put(75,62){$L_s$}
				\end{overpic}
				
				$(b)$
			\end{center}
		\end{minipage}
	\end{center}
	\caption{Examples of a tangential singularity $p$ such that $(a)$ $A_s\neq A_u$ and $(b)$ $A_s=A_u$.}\label{Fig18}
\end{figure}
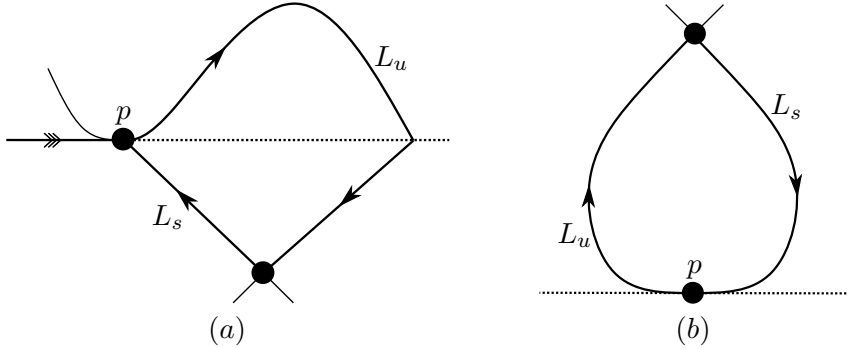	

\begin{definition}
	Given a tangential singularity $p$, let $X_s$ and $X_u$ be as above. We define the stable and unstable contact order of $p$ as the contact order $n_s$ and $n_u$ of $X_s$ and $X_u$ with $\Sigma$ at $p$, respectively. Furthermore we also say that $X_s$ and $X_u$ are the stable and unstable components of $Z$ defined at $p$.
\end{definition}

\begin{definition}
	Let $\Gamma^n$ be an elementary polycycle with distinct singularities $p_1,\dots,p_n$. The hyperboli\-city ratio of $r_i>0$ of $p_i$ is defined as follows.
	\begin{enumerate}[label=(\alph*)]
		\item If $p_i$ is a tangential singularity, then $r_i=\frac{n_{i,u}}{n_{i,s}}$, where $n_{i,s}$ and $n_{i,u}$ are the stable and unstable contact orders of $p_i$.
		\item If $p_i$ is a hyperbolic saddle, then $r_i=\frac{|\nu_i|}{\lambda_i}$, where $\nu_i<0<\lambda_i$ are the eigenvalues of $p_i$.
	\end{enumerate}
\end{definition}

In the case of a smooth vector fields, Cherkas \cite{Cherkas} proved that if $\Gamma$ is a polycycle composed by $n$ hyperbolic saddles $p_1,\dots,p_n$, with hyperbolicity ratios $r_1,\dots,r_n$, then $\Gamma$ is stable if,
	\[r:=\prod_{i=1}^{n}r_i>1,\]
and unstable if $r<1$. Therefore, our first main theorem is an extension, to non-smooth vector fields, of such classic result.

\begin{theorem}\label{Main1}
	Let $Z=(X_1,\dots,X_M;\Sigma)$ be a planar non-smooth vector field with an elementary polycycle $\Gamma^n$. Let also,
	\begin{equation}\label{1}
		r(\Gamma^n)=\prod_{i=1}^{n}r_i.
	\end{equation}
	If $r(\Gamma^n)>1$ (resp. $r(\Gamma^n)<1$), then there is a neighborhood $N_0$ of $\Gamma^n$ such that the orbit of $Z$ through any point $p\in N_0\cap A$ has $\Gamma^n$ as $\omega$-limit (resp. $\alpha$-limit). 
\end{theorem}

\begin{definition}\label{Def4}
	Let $\Gamma^n$ be a semi-elementary polycycle with $n$ distinct singularities $p_1,\dots,p_n$. We say that $p_i$ is stable (unstable) singularity of $\Gamma^n$ if it satisfies one of the following conditions.
	\begin{enumerate}[label=(\alph*)]
		\item $p_i$ is a semi-hyperbolic singularity and $\lambda_i<0$ (resp. $\lambda_i>0$), where $\lambda_i$ is the unique non-zero eigenvalue of $p_i$;
		\item $p_i$ is a hyperbolic saddle and $r_i>1$ (resp. $r_i<1$), where $r_i$ is the hyperbolicity ratio of $p_1$;
		\item $p_i$ is a tangential singularity and $n_{i,s}=1$ (resp. $n_{i,u}=1)$, where $n_{i,s}$ and $n_{i,u}$ are the stable and unstable contact orders of $p_i$.
	\end{enumerate}
\end{definition}

Let $\Gamma^n$ be a semi-elementary polycycle. We say that the \emph{cyclicity} of $\Gamma^n$ is $k$ if at most $k$ limit cycles can bifurcate from a arbitrarily small perturbation of $\Gamma^n$. In the case of smooth vector fields, Dumortier et al \cite{DumRouRou} proved that if $\Gamma$ is a polycycle of a smooth vector field composed only by stable (resp. unstable) singular points, then $\Gamma$ has cyclicity one. Furthermore if any small perturbation of $\Gamma$ has a limit cycle, then it is hyperbolic and stable (resp. unstable). Therefore, our second main theorem is an extension of such result to the realm of non-smooth vector fields.

\begin{theorem}\label{Main3}
	Let $Z=(X_1,\dots,X_M;\Sigma)$ be a planar non-smooth vector field with a semi-elementary polycycle $\Gamma^n$. Suppose that each singularity $p_i$ is a stable (resp. unstable) singularity of $\Gamma^n$. If a small perturbation of $\Gamma^n$ has a limit cycle, then it is unique, hyperbolic and stable (resp. unstable). In particular, the cyclicity of $\Gamma^n$ is one.
\end{theorem}

Let $\Gamma$ be a polycycle of a smooth vector field composed by $n$ hyperbolic saddles $p_1,\dots,p_n$, with hyperbolicity ratios $r_1,\dots,r_n$. Let also, $R_i=\prod_{j=1}^{i}r_j$. Han et al \cite{HanWuBi} proved that if $(R_i-1)(R_{i+1}-1)<0$, $i\in\{1,\dots,n-1\}$, then there exists an arbitrarily small $C^\infty$-perturbation of $\Gamma$ with at least $n$ limit cycles. In our third main result, we extend this result to the case of non-smooth vector fields.

\begin{theorem}\label{Main4}
	Let $Z=(X_1,\dots,X_M;\Sigma)$ be a planar non-smooth planar vector field with an elementary polycycle $\Gamma^n$ and let $R_i=\prod_{j=1}^{i}r_j$, $i\in\{1,\dots,n\}$. Suppose $R_n\neq1$ and, if $n\geqslant2$, suppose $(R_i-1)(R_{i+1}-1)<0$ for $i\in\{1,\dots,n-1\}$. Then, there exist an arbitrarily small perturbation of $Z$ such that at least $n$ limit cycles bifurcates from $\Gamma^n$. In particular, the cyclicity of $\Gamma^n$ is at least $n$.
\end{theorem}

\section{Preliminaries}\label{sec2}

\subsection{Transition map near a hyperbolic saddle}\label{sec2.1}

Let $X_\mu$ be a $C^\infty$ planar vector field depending in a $C^\infty$-way on a parameter $\mu\in\mathbb{R}^r$, $r\geqslant1$, defined in a neighborhood of a hyperbolic saddle $p_0$ at $\mu=\mu_0$. Let $\Lambda\subset\mathbb{R}^r$ be a small enough neighborhood of $\mu_0$, $\nu(\mu)<0<\lambda(\mu)$ be the eigenvalues of the hyperbolic saddle $p(\mu)$, $\mu\in\Lambda$, and $r(\mu)=\frac{|\nu(\mu)|}{\lambda(\mu)}$ be the hyperbolicity ratio of $p(\mu)$. Let $B$ be a small enough neighborhood of $p_0$ and $\Phi:B\times\Lambda\to\mathbb{R}^2$ be a $C^\infty$-change of coordinates such that $\Phi$ sends the hyperbolic saddle $p(\mu)$ to the origin and its unstable and stable manifolds $W^u(\mu)$ and $W^s(\mu)$ to the axis $Ox$ and $Oy$, respectively. Let $\sigma$ and $\tau$ be two small enough cross sections of $Oy^+$ and $Ox^+$, respectively. We can suppose that $\sigma$ and $\tau$ are parametrized by $x\in[0,x_0]$ and $y\in[0,y_0]$, with $x=0$ and $y=0$ corresponding to $Oy^+\cap\sigma$ and $Ox^+\cap\tau$, respectively. The flow of $X_\mu$ in the first quadrant in this new coordinate system defines a transition map:
	\[D\colon(0,x_0]\times\Lambda\to(0,y_0],\]
called the \emph{Dulac's map} \cite{Dulac1923}. See Figure~\ref{Fig16}. Observe that $D$ is of class $C^\infty$ for $x\neq0$ and it can be continuously extend by $D(0,\mu)=0$ for all $\mu\in\Lambda$.
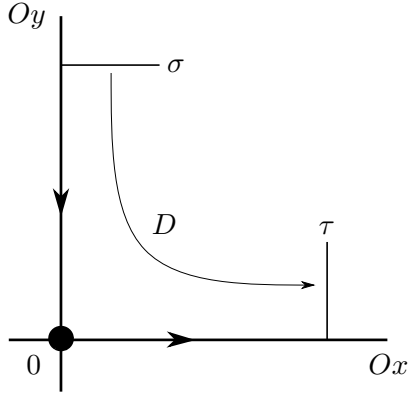
\begin{figure}[ht]
	\begin{center}
		\begin{overpic}[height=5cm]{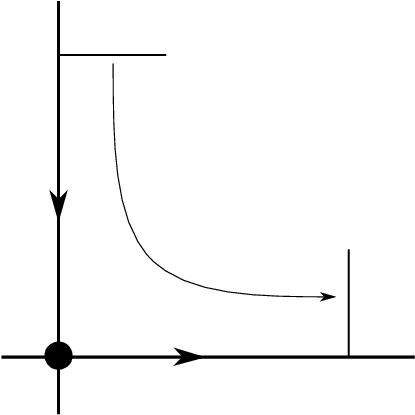} 
			\put(42,85){$\sigma$}
			\put(82,42){$\tau$}
			\put(5,5){$0$}
			\put(0,98){$Oy$}
			\put(95,5){$Ox$}
			\put(38,42){$D$}
		\end{overpic}
	\end{center}
	\caption{The Dulac map near a hyperbolic saddle.}\label{Fig16}
\end{figure}

\begin{definition}
	Let $I_k$, $k\geqslant0$, denote the set of functions $f:[0,x_0]\times\Lambda_k\to\mathbb{R}$, with $\Lambda_k\subset\Lambda$, satisfying the following properties.
	\begin{enumerate}[label=(\alph*)]
		\item $f$ is $C^\infty$ on $(0,x_0]\times \Lambda_k$;
		\item For each $j\in\{0,\dots,k\}$ we have that $\varphi_j=x^j\frac{\partial^j f}{\partial x^j}(x,\mu)$ is continuous on $(0,x_0]\times\Lambda_k$ with $\varphi_j(x,\mu)\to0$ for $x\to0$, uniformly in $\mu$.
	\end{enumerate}
	A function $f:[0,x_0]\times\Lambda\to\mathbb{R}$ is said to be of class $I$ if $f$ is $C^\infty$ on $(0,x_0]\times\Lambda$ and for every $k\geqslant0$ there exists a neighborhood $\Lambda_k\subset\Lambda$ of $\mu_0$ such that $f$ is of class $I^k$ on $(0,x_0]\times\Lambda_k$.
\end{definition}

\begin{theorem}[Mourtada, \cites{Mourtada,DumRouRou}]
	Let $X_\mu$, $\sigma$, $\tau$, and $D$ be as above. Then, for $(x,\mu)\in(0,x_0]\times\Lambda$, we have
\begin{equation}\label{MourtadaForm}
	D(x,\mu)=x^{r(\mu)}(A(\mu)+\varphi(x,\mu)),
\end{equation}
	with $\varphi\in I$ and $A$ a positive $C^\infty$-function.
\end{theorem}

Following Dumortier et al \cite{DumRouRou}, we call \emph{Mourtada's form} the expression \eqref{MourtadaForm} of the Dulac map and denote by $\mathfrak{D}$ the class of maps given by \eqref{MourtadaForm}.

\begin{proposition}[\cites{Mourtada,DumRouRou}]
	Given $D(x,\mu)=x^{r(\mu)}(A(\mu)+\varphi(x,\mu))\in\mathfrak{D}$, the following statements hold.
	\begin{enumerate}[label=(\alph*)]
		\item $D^{-1}$ is well defined and $D^{-1}(x,\mu)=x^{\frac{1}{r(\mu)}}(B(\mu)+\psi(x,\mu))\in\mathfrak{D}$;
		\item $\frac{\partial D}{\partial x}$ is well defined and 
		\begin{equation}\label{2}
			\frac{\partial D}{\partial x}(x,\mu)=r(\mu)x^{r(\mu)-1}(A(\mu)+\xi(x,\mu)),
		\end{equation}
		with $\xi\in I$.
	\end{enumerate}
\end{proposition}

For a complete characterization of the Dulac map, see \cites{MarVil2020,MarVil2021}. The following result is also a classical result about the Dulac map.

\begin{proposition}\label{P2}
	Let $X$ be a vector field of class $C^\infty$ with a hyperbolic saddle $p$ at the origin, with eigenvalues $\nu<0<\lambda$. Suppose also that the unstable and stable manifolds $W^u$ and $W^s$ of $p$ are given by the axis $Oy$ and $Ox$, respectively, and let $D=D(x)$ be the Dulac map associated with $p$. Then, given $\varepsilon>0$, there is $\delta>0$ such that,
		\[\delta^{1-\frac{|\nu|}{\lambda-\varepsilon}}x^{\frac{|\nu|}{\lambda-\varepsilon}}<D(x)<\delta^{1-\frac{|\nu|}{\lambda+\varepsilon}}x^{\frac{|\nu|}{\lambda+\varepsilon}}.\]	
\end{proposition}

The proof of Proposition~\ref{P2} is due to Sotomayor \cite[Section~$2.2$]{Soto}. A similar result was also proved by Cherkas \cite{Cherkas}. Since both the references are not in English (and as far as we know, there are no translation of it), we find it useful to prove Proposition~\ref{P2} in this paper.

\noindent\textit{Proof of Proposition~\ref{P2}.} Let $X=(P,Q)$ be given by,
	\[P(x,y)=\lambda x+r_1(x,y), \quad Q(x,y)=\nu y+ r_2(x,y).\]
Since $W^s$ and $W^u$ are given the coordinate axis, it follows that $r_1(0,y)=r_2(x,0)=0$, for every $(x,y)\in\mathbb{R}^2$. Observe that,
	\[r_1(x,y)=r_1(0,y)+x\int_{0}^{1}\frac{\partial r_1}{\partial x}(sx,y)\;ds.\]
Hence, it follows that we can write $r_1(x,y)=x\overline{r}_1(x,y)$, with $\overline{r}_1$ continuous and such that $\overline{r}_1(0,0)=0$. Similarly, we have $r_2(x,y)=y\overline{r}_2(x,y)$. Given $\varepsilon>0$, consider the linear vector field $X_\varepsilon=(P_\varepsilon,Q_\varepsilon)$ given by,
	\[P_\varepsilon(x,y)=(\lambda+\varepsilon)x, \quad Q_\varepsilon(x,y)=\nu y.\]
Let
	\[J(x,y)=\left(\begin{array}{cc} P(x,y) & P_\varepsilon(x,y) \vspace{0.2cm} \\ Q(x,y) & Q_\varepsilon(x,y) \end{array}\right)=\left(\begin{array}{cc} \lambda x+x\overline{r}_1(x,y) & (\lambda+\varepsilon)x \vspace{0.2cm} \\ \nu y+y\overline{r}_2(x,y) & \nu y \end{array}\right),\]
and observe that
	\[\det J(x,y)=xy\bigl(\nu\overline{r}_1(x,y)-\nu\varepsilon-(\lambda+\varepsilon)\overline{r}_2(x,y)\bigr).\]
Therefore, there is $\delta>0$ such that if $0<x<\delta$ and $0<y<\delta$, then $\det J(x,y)>0$ and thus the vectors $X(x,y)$ and $X_\varepsilon(x,y)$ have positive orientation. Hence, if $D_\varepsilon$ is the Dulac map associated with $X_\varepsilon$, it follows that $D(x)\leqslant D_\varepsilon(x)$, for every $0<x<\delta$. See Figure~\ref{Fig21}.
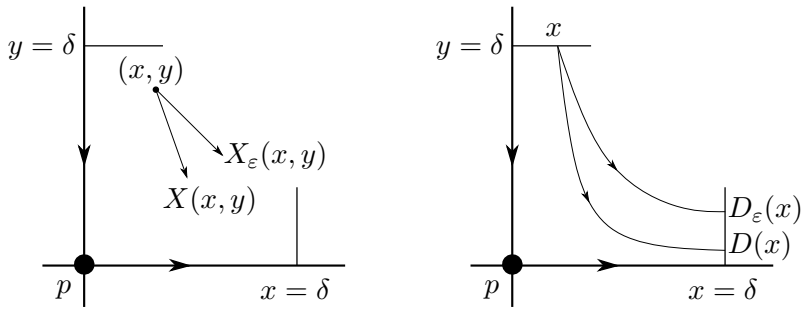
\begin{figure}[ht]
	\begin{center}
		\begin{minipage}{5.5cm}
			\begin{center}
				\begin{overpic}[height=4cm]{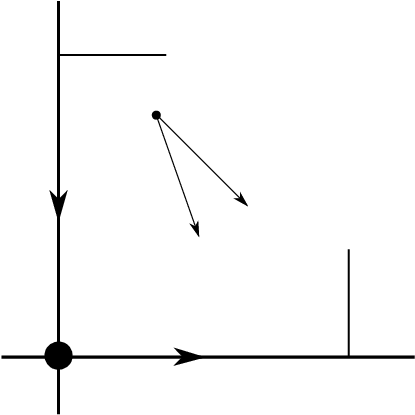} 
					\put(-11,84){$y=\delta$}
					\put(72,3){$x=\delta$}
					\put(5,5){$p$}
					\put(25,76){$(x,y)$}
					\put(60,48){$X_\varepsilon(x,y)$}
					\put(40,33){$X(x,y)$}
				\end{overpic}
			\end{center}
		\end{minipage}
		\begin{minipage}{5.5cm}
			\begin{center}
				\begin{overpic}[height=4cm]{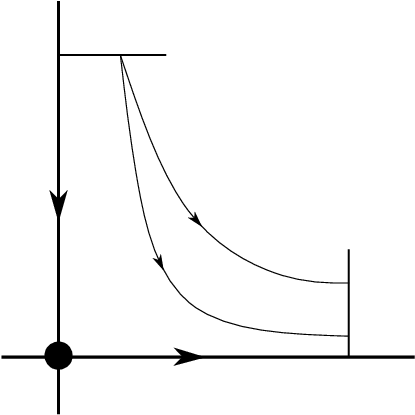} 
					\put(-11,84){$y=\delta$}
					\put(72,3){$x=\delta$}
					\put(5,5){$p$}
					\put(25,90){$x$}
					\put(85,30){$D_\varepsilon(x)$}
					\put(85,18){$D(x)$}
				\end{overpic}
			\end{center}
		\end{minipage}
	\end{center}
	\caption{Illustration of $X$, $X_\varepsilon$, $D$ and $D_\varepsilon$.}\label{Fig21}
\end{figure}
Since $X_\varepsilon$ is linear, it follows that its flow $\varphi$ is given by
	\[\varphi(t,x,y)=\bigl(xe^{(\lambda+\varepsilon)t},ye^{\nu t}\bigr),\]
and thus we have $D_\varepsilon(x)=ye^{\nu t_0}$, where $t_0>0$ is such that $xe^{(\lambda+\varepsilon)t_0}=\delta$. Hence, $t_0=\frac{1}{\lambda+\varepsilon}\ln\frac{\delta}{x}$ and thus we have,
	\[D_\varepsilon(x)=\delta\left(\exp\ln\frac{\delta}{x}\right)^{\frac{\nu}{\lambda+\varepsilon}}=\delta^{1-\frac{|\nu|}{\lambda-\varepsilon}}x^{\frac{|\nu|}{\lambda-\varepsilon}}.\]
This proves the first inequality of the proposition. The other one can be obtained by considering $X_\varepsilon(x,y)=((\lambda-\varepsilon)x,\nu y)$. {\hfill$\square$}

\subsection{Transition map near a semi-hyperbolic singularity}\label{sec2.2}

\begin{theorem}[Theorem~$3.2.2$ of \cite{DumRouRou}]\label{Semihyperbolic}
	Let $X_\mu$ be a $C^\infty$-planar vector field depending in a $C^\infty$-way on a parameter $\mu\in\Lambda\subset\mathbb{R}^r$, $r\geqslant 1$. Suppose that at $\mu=\mu_0$ we have a semi-hyperbolic singularity at the origin $O$. Let also $B$ be a small enough neighborhood of $O$. If $\Lambda$ is a small enough neighborhood of $\mu_0$, then for each $k\geqslant1$, $k\in\mathbb{N}$, there exists a $C^k$-family of diffeomorphisms on $B$ such that at this new coordinate system, $X_\mu$ is given by
		\[\dot x = g(x,\mu), \quad \dot y = \pm y,\]
	except by the multiplication of a $C^k$-positive function. Furthermore, $g$ is a function of class $C^k$ satisfying,
		\[g(0,\mu_0)=\frac{\partial g}{\partial x}(0,\mu_0)=0.\]
\end{theorem}

Let $X_\mu$ be a $C^\infty$-planar vector field depending in a $C^\infty$-way on a parameter $\mu\in\Lambda\subset\mathbb{R}^r$, $r\geqslant 1$. Suppose that at $\mu=\mu_0$ we have a semi-hyperbolic singularity $p_0$ with a hyperbolic sector (e.g. a saddle-node or a degenerated saddle). At $\mu=\mu_0$, let $\lambda\in\mathbb{R}\backslash\{0\}$ be the unique non-zero eigenvalue of $p_0$. Reversing the time if necessary, we can assume that $\lambda<0$. Locally at $p_0$, it follows from Theorem~\ref{Semihyperbolic} that we can suppose that $X_\mu$ is given by
	\[\dot x = g(x,\mu), \quad \dot y=-y,\]
with $g$ of class $C^k$ (for any $k$ large enough) and satisfying,
	\[g(0,\mu_0)=\frac{\partial g}{\partial x}(0,\mu_0)=0.\]
In this new coordinate system given by Theorem~\ref{Semihyperbolic}, and at $\mu=\mu_0$, let $\sigma$ and $\tau$ be two small cross sections of the axis $Oy^+$ and $Ox^+$ (which are, respectively, the stable and the central manifolds of $p_0$). As in subsection~\ref{sec2.1}, we can suppose that $\sigma$ and $\tau$ are parametrized by $x\in[0,x_0]$ and $y\in[0,y_0]$, with $x=0$ and $y=0$ corresponding to $Oy^+\cap\sigma$ and $Ox^+\cap\tau$, respectively (see Figure~\ref{Fig16}). Let $x^*(\mu)$ be the largest solution of $g(x,\mu)=0$ and observe that $x^*(\mu_0)=0$. For each $\mu\in\Lambda$, let $\sigma(\mu)\subset\sigma$ be given by $x\in[x^*(\mu),x_0]$ and let
	\[C=\bigcup_{\mu\in\Lambda}\{\sigma(\mu),\mu\}\subset\sigma\times\Lambda.\]
As in Section~\ref{sec2.1}, in this new coordinate system the flow of $X_\mu$ defines a transition map $F\colon C\to(0,y_0]$.
	
\begin{theorem}[Theorem~$3$ of \cite{DumRouRou}]
	Let $X_\mu$ and $F$ be as above. Then
	\begin{equation}\label{0}
		F(x,\mu)=Ye^{-T(x,\mu)},
	\end{equation}
	where $Y>0$ and $T\colon C\to\mathbb{R}^+$ is the time function from $\sigma(\mu)$ to $\tau$. Moreover, if $\mu=(\mu_1,\dots,\mu_r)$, then for any $k$, $m\in\mathbb{N}$ and for any $(i_0,\dots,i_r)\in\mathbb{N}^{n+1}$ with $i_0+\dots+i_r=m$, we have
	\begin{equation}\label{3}
		\frac{\partial^m F}{\partial x^{i_0}\partial\mu_1^{i_1}\dots\partial\mu_r^{i_r}}(x,\mu)=O(||(x,\mu)||^k).
	\end{equation}
\end{theorem}

\subsection{Transition map near a tangential singularity}\label{sec2.3}

Let $p_0$ be a tangential singularity of $\Gamma^n$ and $X_s$, $X_u$ be the stable and unstable components of $Z$ defined at $p_0$ with $\mu=\mu_0$. Let $B$ be a small enough neighborhood of $p_0$ and $\Phi:B\times\Lambda\to\mathbb{R}^2$ be a $C^\infty$ change of coordinates such that $\Phi(p_0,\mu_0)=(0,0)$ and $\Phi(B\cap\Sigma)=Ox$. Let $l_s=\Phi(B\cap L_s)$, $l_u=\Phi(B\cap L_u)$ and $\tau_s$, $\tau_u$ two small enough cross sections of $l_s$ and $l_u$, respectively. Let also,
	\[\sigma=[0,\varepsilon)\times\{0\}, \quad \sigma=(-\varepsilon,0]\times\{0\}, \text{ or } \sigma=\{0\}\times[0,\varepsilon),\]
depending on $\Gamma^n$. It follows from Andrade et al \cite{AndGomNov} that $\Phi$ can be choose such that the transition maps $T^{s,u}\colon\sigma\times\Lambda\to\tau_{s,u}$, given by the flow of $X_{s,u}$ in this new coordinate system, are well defined and given by,
\begin{equation}\label{26}
	\begin{array}{l}
	\displaystyle T^u(h_\mu(x),\mu)=k_u(\mu)x^{n_u}+O(x^{n_u+1})+\sum_{i=0}^{n_u-2}\lambda_i^u(\mu)x_i, \vspace{0.2cm} \\ \displaystyle T^s(h_\mu(x),\mu)=k_s(\mu)x^{n_s}+O(x^{n_s+1})+\sum_{i=0}^{n_s-2}\lambda_i^s(\mu)x_i,
	\end{array}
\end{equation}
with $\lambda_i^{s,u}(\mu_0)=0$, $k_{s,u}(\mu_0)\neq0$, $h_\mu\colon\mathbb{R}\to\mathbb{R}$ a diffeomorphism and with $h_\mu$ and $\lambda_i^{s,u}$ depending continuously on $\mu$. For examples of such maps, see Figure~\ref{Fig17}.
\begin{figure}[ht]
	\begin{center}
		\begin{minipage}{4cm}
			\begin{center}
				\begin{overpic}[width=3.8cm]{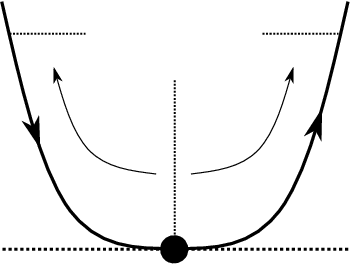} 
					\put(62,35){$T^u$}
					\put(30,35){$T^s$}
				\end{overpic}
				
				$(a)$
			\end{center}
		\end{minipage}
		\begin{minipage}{4cm}
			\begin{center}
				\begin{overpic}[width=3.8cm]{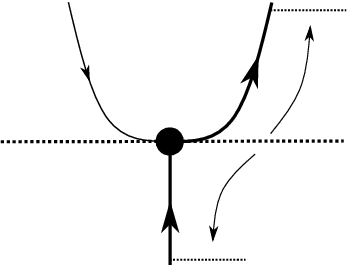} 
					\put(90,50){$T^u$}
					\put(70,15){$T^s$}
				\end{overpic}
				
				$(b)$
			\end{center}
		\end{minipage}
		\begin{minipage}{4cm}
			\begin{center}
				\begin{overpic}[width=3.8cm]{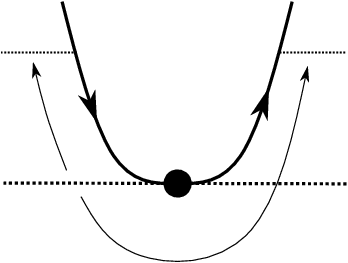} 
					\put(80,5){$T^u$}
					\put(0,30){$T^s$}
				\end{overpic}
		
				$(c)$
			\end{center}
	\end{minipage}
	\end{center}
\caption{Illustration of the maps $T^u$ and $T^s$. The choice between $(a)$ and $(c)$ depends on whether the Poincar\'e map is defined in the bounded or unbounded region delimited by $\Gamma^n$.}\label{Fig17}
\end{figure}
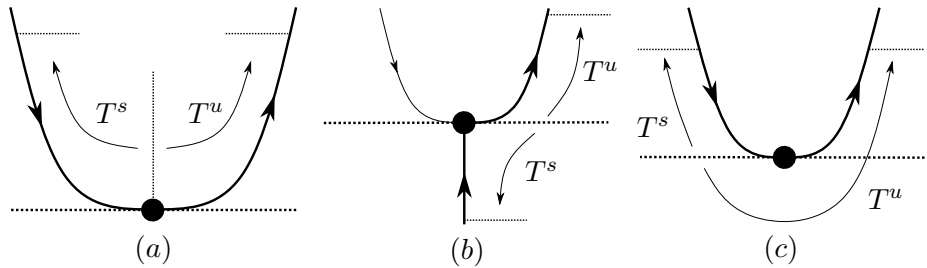	

\section{Proofs of Theorems~\ref{Main1} and \ref{Main3}}\label{sec3}

\noindent {\it Proof of Theorem~\ref{Main1}.} For simplicity, we assume that $\Sigma=h^{-1}(0)$ has one component and thus $Z=(X_1,X_2;\Sigma)$ has two components. Moreover, we assume that $\Gamma^n=\Gamma^3$ is composed by two tangential singularities $p_1$, $p_2$ and by a hyperbolic saddle $p_3$. See Figure~\ref{Fig2}. The general case follows similarly. Let $B_i$ be a small enough neighborhood of $p_i$ and let $\Phi_i\colon B_i\times\{\mu_0\}\to\mathbb{R}^2$ be the change of variables chosen as in Section~\ref{sec2.3}, $i\in\{1,2\}$. Let also $B_3$ be a neighborhood of $p_3$ and $\Phi_3\colon B_3\times\{\mu_0\}\to\mathbb{R}^2$ be the change of variables chosen as in Section~\ref{sec2.1}. Knowing that $T_i^{s,u}\colon\sigma_i\times\{\mu_0\}\to\tau_i^{s,u}$ and $D\colon\sigma\times\{\mu_0\}\to \tau$, let,
	\[\overline{\sigma}_i=\Phi_i^{-1}(\sigma_i), \quad \overline{\tau}_i^s=\Phi_i^{-1}(\tau_i^s), \quad \overline{\tau}_i^u=\Phi_i^{-1}(\tau_i^u),\]
	\[J_s=\Phi_3^{-1}(\sigma), \quad J_u=\Phi_3^{-1}(\tau),\]
with $i\in\{1,2\}$. Let also, 
	\[\overline{\rho}_1:\tau_1^u\to\tau_2^s, \quad \overline{\rho}_2:\tau_2^u\to J_s, \quad\overline{\rho_3}:J_u\to\tau_1^s,\]
be defined by the flow of $X_1$ and $X_2$. See Figure~\ref{Fig2}. Finally let, 
	\[\rho_1=\Phi_2\circ\overline{\rho}_1\circ\Phi_1^{-1}, \quad \rho_2=\Phi_3\circ\overline{\rho}_2\circ\Phi_2^{-1}, \quad \rho_3=\Phi_1\circ\overline{\rho}_3\circ\Phi_3^{-1},\]
and,
	\[\overline{T}_i^s=\Phi_i^{-1}\circ T_i^s\circ\Phi_i, \quad \overline{T}_i^u=\Phi_i^{-1}\circ T_i^u\circ\Phi_i, \quad \overline{D}=\Phi_3^{-1}\circ D\circ\Phi_3,\]
with $i\in\{1,2\}$. See Figure~\ref{Fig2}.
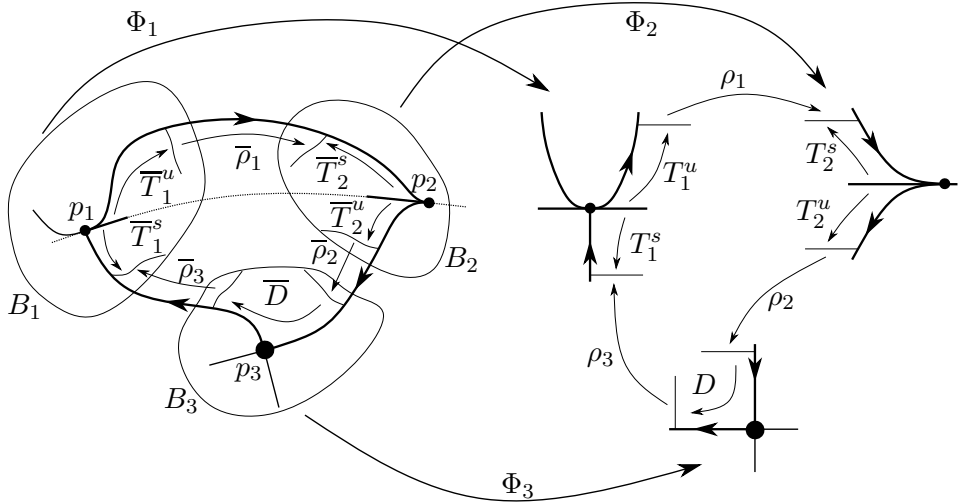
\begin{figure}[ht]
\begin{center}
\begin{overpic}[width=12.5cm]{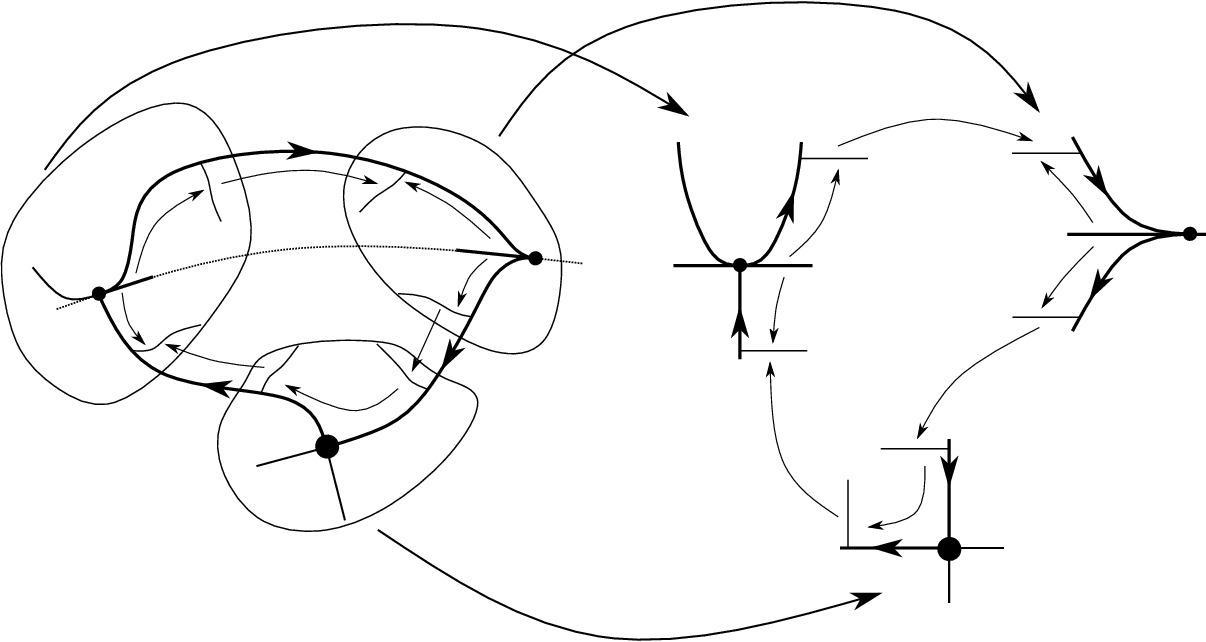} 
	\put(0,20){$B_1$}
	\put(46,25){$B_2$}
	\put(16.5,10){$B_3$}
	\put(12.5,50){$\Phi_1$}
	\put(65,50){$\Phi_2$}
	\put(52,1){$\Phi_3$}
	\put(24,36){$\overline{\rho}_1$}
	\put(32,26){$\overline{\rho}_2$}
	\put(18,24){$\overline{\rho}_3$}
	\put(14,32.5){$\overline{T}_1^u$}
	\put(13,27.25){$\overline{T}_1^s$}
	\put(32.5,34){$\overline{T}_2^s$}
	\put(34,29.25){$\overline{T}_2^u$}
	\put(27,21){$\overline{D}$}
	\put(6.5,30.5){$p_1$}
	\put(42.5,33.5){$p_2$}
	\put(24,13.5){$p_3$}
	\put(72,11.5){$D$}
	\put(75,44.5){$\rho_1$}
	\put(80,21){$\rho_2$}
	\put(61,15){$\rho_3$}
	\put(69,34){$T_1^u$}
	\put(65.5,26.5){$T_1^s$}
	\put(84,36){$T_2^s$}
	\put(83,30){$T_2^u$}
\end{overpic}
\end{center}
\caption{Illustration of the maps used in the proof of Theorem~\ref{Main1}.}\label{Fig2}
\end{figure}
Let $\nu<0<\lambda$ be the eigenvalues of $p_3$ and denote $r=\frac{|\nu|}{\lambda}$. Let also $n_{i,s}$ and $n_{i,u}$ denote the stable and unstable order of $p_i$, $i\in\{1,2\}$. Suppose $r(\Gamma^n)>1$. Given $\varepsilon>0$, it follows from Sections~\ref{sec2.1} and \ref{sec2.3} that,
	\[T_i^s(x)=k_{i,s}x^{n_{i,s}}+O(x^{n_{i,s}+1}), \quad T_i^u(x)=k_{i,u}x^{n_{i,u}}+O(x^{n_{i,u}+1}),\]
	\[D(x)<Cx^{\frac{|\nu|}{\lambda+\varepsilon}}, \quad \rho_j(x)=a_jx+O(x^2),\]
with $k_{i,s}$, $k_{i,u}$, $a_j$, $a\neq0$, $C>0$, $i\in\{1,2\}$ and $j\in\{1,2,3\}$. Since $\varepsilon>0$ is arbitrary, it follows that if we define
	\[\pi=\rho_2\circ T_2^u\circ(T_2^s)^{-1}\circ\rho_1\circ T_1^u\circ(T_1^s)^{-1}\circ\rho_3\circ D,\]
then one can conclude that,
	\[\pi(x)\leqslant Kx^{r_0}+O(x^{n_0+1}),\]
with $K\neq0$ and $1<r_0<r(\Gamma^n)$. Hence, if $x$ is small enough we conclude that $\pi(x)<x$. The result now follows from the fact that the first return map
	\[P=\overline{\rho}_2\circ\overline{T}_2^u\circ\left(\overline{T}_2^s\right)^{-1}\circ\overline{\rho}_1\circ\overline{T}_1^u\circ\left(\overline{T}_1^s\right)^{-1}\circ\overline{\rho}_3\circ\overline{D},\]
satisfies $P=\Phi_3^{-1}\circ\pi\circ\Phi_3$. If $r(\Gamma^n)<1$, the results follows by inverting the time variable. {\hfill$\square$}

\noindent {\it Proof of Theorem~\ref{Main3}.} Let us suppose that every singularity of $\Gamma^n$ is attracting (see Definition~\ref{Def4}). Following the proof of Theorem~\ref{Main1}, we observe that the Poincar\'e map, when well defined, can be written as the composition
	\[P_\mu=G_k\circ F_k \circ \dots \circ G_1\circ F_1,\]
where each $F_i$ is the transition map near a hyperbolic saddle (given by \eqref{MourtadaForm}), a semi-hyperbolic singularity (given by \eqref{0}), or a tangential singularity (given by \eqref{26}), and each $G_i$ is regular transition given by the flow of $Z$, i.e. a $C^\infty$-diffeomorphism in $x$. We call $y_1=F_1(x_1)$, $x_2=G_1(y_1)$, $\dots$, $y_k=F_k(x_k)$, $x_{k+1}=G_k(y_k)$. Thus,
	\[P_\mu'(x_1)=G_k'(y_k)F_k'(x_k)\dots G_1'(y_1)F_1'(x_1).\]
Therefore, it follows from \eqref{2}, \eqref{3} and \eqref{26} that for all $\varepsilon>0$ there exists a neighborhood $\Lambda$ of $\mu_0$ and neighborhoods $W_i$ of $x_i=0$, $i\in\{1,\dots,k+1\}$, such that if $x_1\in W_1$, then $x_i\in W_i$ and $|F_i'(x_i)|<\varepsilon$, for all $i\in\{1,\dots,k+1\}$ and for all $\mu\in\Lambda$. Also, if $\Lambda$ and each $W_i$ are small enough, then each $G_i'(y_i)$ is bounded, and bounded away from zero. Since $\varepsilon>0$ is arbitrarily small, it follows that $P_\mu'(x_1)$ is also arbitrarily small, for $(x_1,\mu)\in W_1\times \Lambda$. Therefore, the derivative of the displacement map $d_\mu(x_1)=P_\mu(x_1)-x_1$ cannot vanish and thus at most one limit cycle bifurcate from $\Gamma^n$. Moreover, if it does, it is hyperbolic and stable. The other case follows by reversing the time variable. {\hfill$\square$}

Unlike Theorems~\ref{Main1} and \ref{Main3}, to obtain Theorem~\ref{Main4} it will be necessary to work on some technicalities about the displacement maps of a polycycle. We will deal with that at Sections~\ref{sec4} and \ref{sec5}.

\section{The displacement map}\label{sec4}

Let $Z=(X_1,\dots,X_m;\Sigma)$ be a planar non-smooth vector field, depending in a $C^\infty$-way on a parameter $\mu\in\mathbb{R}^r$, and such that $Z$ has an elementary polycycle $\Gamma^n$ at $\mu=\mu_0$. Let also $\Lambda\subset\mathbb{R}^r$ be a small neighborhood of $\mu_0$ and from now on assume $\mu\in\Lambda$. In this section, we will study the displacement map between two singularities $p_i$ and $p_{i+1}$ of $\Gamma^n$. We will begin by the case in which both $p_i$ and $p_{i+1}$ are hyperbolic saddles. To simplify the notation, at $\mu=\mu_0$, let $p_1\in A_1$ and $p_2\in A_2$ be two hyperbolic saddles of $\Gamma^n$ with the heteroclinic connection $L_0$ such that $\omega(L_0)=p_1$, $\alpha(L_0)=p_2$ and $L_0\cap\Sigma=\{x_0\}$, $\Sigma=h^{-1}(0)$. Let $\gamma_0(t)$ be a parametrization of $L_0$ such that $\gamma_0(0)=x_0$ and $u_0$ be an unitary vector orthogonal to $T_{x_0}\Sigma$ such that $sign(\left<u_0,\nabla h(x_0)\right>)=sign(X_1h(x_0))=sign(X_2h(x_0))$. See Fi\-gure~\ref{Fig3}.
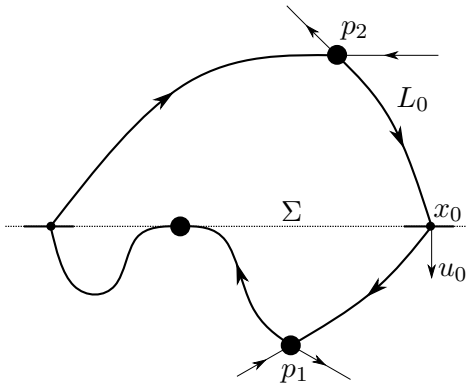
\begin{figure}[ht]
	\begin{center}
		\begin{overpic}[height=5cm]{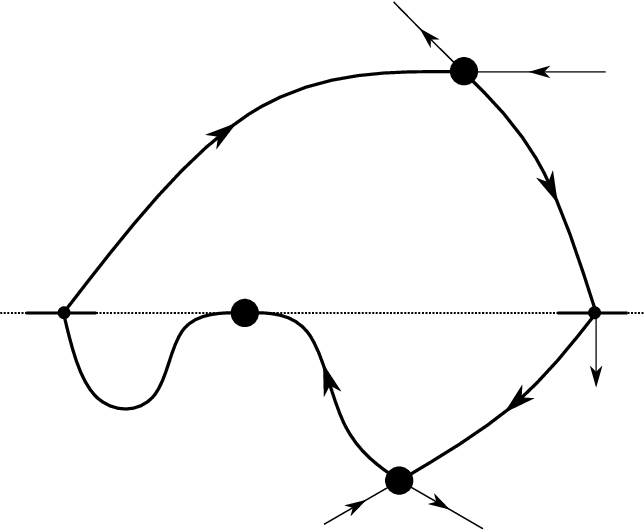} 
			\put(60,1){$p_1$}
			\put(73,76){$p_2$}
			\put(93,36){$x_0$}
			\put(94,23){$u_0$}
			\put(60,35){$\Sigma$}
			\put(85,60){$L_0$}
		\end{overpic}
	\end{center}
	\caption{Illustration of $\Gamma^n$ at $\mu=\mu_0$.}\label{Fig3}
\end{figure}
Also, define $\omega_0\in\{-1,1\}$ such that $\omega_0=1$ if the orientation of $\Gamma^n$ is counterclockwise and $\omega_0=-1$ if the orientation of $\Gamma^n$ is clockwise. We denote by $DX(p,\mu^*)$ the Jacobian matrix of $X|_{\mu=\mu^*}$ at $p$, i.e. if $X=(P,Q)$, then
	\[DX(p,\mu^*)=\left(\begin{array}{cc} \dfrac{\partial P}{\partial x_1}(p,\mu^*) & \dfrac{\partial P}{\partial x_2}(p,\mu^*) \vspace{0.2cm} \\  \dfrac{\partial Q}{\partial x_1}(p,\mu^*) & \dfrac{\partial Q}{\partial x_2}(p,\mu^*) \end{array}\right).\]
If $\Lambda$ is a small enough neighborhood of $\mu_0$, then it follows from the Implicit Function Theorem that if $\mu\in\Lambda$, then the perturbation $p_i(\mu)$ of $p_i$ is well defined and it is a hyperbolic saddle of $X_i$, with $p_i(\mu)\to p_i$ as $\mu\to\mu_0$, and with $p_i(\mu)$ of class $C^\infty$, $i\in\{1,2\}$. Let $(y_{i,1},y_{i,2})=(y_{i,1}(\mu),y_{i,2}(\mu))$ be a coordinate system with its origin at $p_i(\mu)$ and such that the $y_{i,1}$-axis and the $y_{i,2}$-axis are the one-dimensional stable and unstable spaces $E_i^s(\mu)$ and $E_i^u(\mu)$ of the linearization of $X_i(\cdot,\mu)$ at $p_i(\mu)$, $i\in\{1,2\}$. It follows from the Center-Stable Manifold Theorem (see \cite{Kel1967}) that the stable and unstable manifolds $S_i^\mu$ and $U_i^\mu$ of $X_i(\cdot,\mu)$ at $p_i(\mu)$ are given by,
	\[S_i^\mu: y_{i,2}=\Psi_{i,2}(y_{i,1},\mu), \quad U_i^\mu: y_{i,1}=\Psi_{i,1}(y_{i,2},\mu),\]
where $\Psi_{i,1}$ and $\Psi_{i,2}$ are $C^\infty$-functions, $i\in\{1,2\}$. Restricting $\Lambda$ if necessary, it follows that there exist $\delta>0$ such that, 
	\[y_i^s(\mu)=(\delta,\Psi_{i,2}(\delta,\mu))\in S_i^\mu, \quad y_i^u(\mu)=(\Psi_{i,1}(\delta,\mu),\delta)\in U_i^\mu,\]
$i\in\{1,2\}$. If $C_i(\mu)$ is the diagonalization of $DX_i(p_i(\mu),\mu)$, then at the original coordinate system $(x_1,x_2)$ we obtain,
	\[x_i^s(\mu)=p_i(\mu)+C_i(\mu)y_i^s(\mu)\in S_i^\mu, \quad x_i^u(\mu)=p_i(\mu)+C_i(\mu)y_i^u(\mu)\in U_i^\mu,\]
$i\in\{1,2\}$. Furthermore $x_i^s(\mu)$ and $x_i^u(\mu)$ are also $C^\infty$ at $\Lambda$. Let $\phi_i(t,\xi,\mu)$ be the flow of $X_i(\cdot,\mu)$ such that $\phi_i(0,\xi,\mu)=\xi$ and $L^s_0=L^s_0(\mu)$, $L^u_0=L^u_0(\mu)$ be the perturbations of $L_0$ such that $\omega(L^s_0(\mu))=p_1(\mu)$ and $\alpha(L^u_0(\mu))=p_2(\mu)$. Then it follows that,
	\[x^s(t,\mu)=\phi_1(t,x_1^s(\mu),\mu), \quad  x^u(t,\mu)=\phi_2(t,x_2^u(\mu),\mu)\]
are parametrizations of $L^s_0(\mu)$ and $L^u_0(\mu)$, respectively. Since $L_0$ intersects $\Sigma$, it follows that there are $t_0^s<0$ and $t_0^u>0$ such that $x^s(t_0^s,\mu_0)=x_0=x^u(t_0^u,\mu_0)$ and thus by the uniqueness of solutions we have,
	\[x^s(t+t_0^s,\mu_0)=\gamma_0(t)=x^u(t+t_0^u,\mu_0),\]
for $t\in[0,+\infty)$ and $t\in(-\infty,0]$, respectively. 

\begin{lemma}\label{Lemma1}
	Taking a small enough neighborhood $\Lambda$ of $\mu_0$, there exists unique $C^\infty$-functions $\tau^s(\mu)$ and $\tau^u(\mu)$ such that $\tau^s(\mu)\to t_0^s$ and $\tau^u(\mu)\to t_0^u$, as $\mu\to\mu_0$, and $x_0^s(\mu)=x^s(\tau^s(\mu),\mu)\in\Sigma$ and $x_0^u(\mu)=x^u(\tau^u(\mu),\mu)\in\Sigma$, for all $\mu\in\Lambda$. See Figure~\ref{Fig4}.
\end{lemma}
\begin{figure}[ht]
	\begin{center}
		\begin{overpic}[height=5cm]{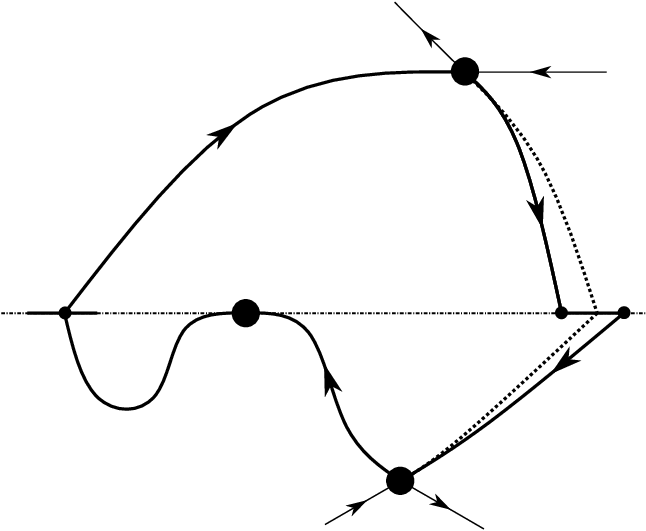} 
			\put(60,1){$p_1$}
			\put(73,76){$p_2$}
			\put(95,37){$x_0^s$}
			\put(78,28){$x_0^u$}
			\put(60,35){$\Sigma$}
			\put(73,50){$L_0^u$}
			\put(82,15){$L_0^s$}
		\end{overpic}
	\end{center}
	\caption{Illustration of $x_0^s(\mu)$ and $x_0^u(\mu)$.}\label{Fig4}
\end{figure}
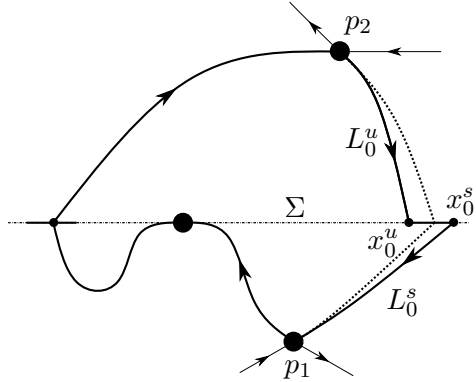
\begin{proof} Let $X_1$ denote a $C^\infty$-extension of $X_1$ to a neighborhood of $\overline{A_1}$ and observe that now $x^s(t,\mu_0)$ is well defined for $|t-t_0^s|$ small enough. Knowing that $\Sigma=h^{-1}(0)$, define $S(t,\mu)=h(x^s(t,\mu))$ and observe that $S(t_0^s,\mu_0)=h(x_0)=0$ and,
	\[\frac{\partial S}{\partial t}(t_0^s,\mu_0)=\left<\nabla h(x_0),X_1(x_0)\right>\neq0.\]
It then follows from the Implicit Function Theorem that there exist a $C^\infty$-function $\tau^s(\mu)$ such that $\tau^s(\mu_0)=t_0^s$ and $S(\tau^s(\mu),\mu)=0$ and thus $x_0^s(\mu)=x^s(\tau^s(\mu),\mu)\in\Sigma$. In the same way one can prove the existence of $\tau^u$. \end{proof}

\begin{definition}\label{DefDisplacement1}
	It follows from lemma~\ref{Lemma1} that the displacement function,
		\[d(\mu)=\omega_0[x_0^u(\mu)-x_0^s(\mu)]\land u_0,\]
	where $(x_1,x_2)\land(y_1,y_2)=x_1y_2-y_1x_2$, is well defined near $\mu_0$ and it is of class $C^\infty$. See Figure~\ref{Fig5}.
\end{definition}
\begin{figure}[ht]
	\begin{center}
		\begin{overpic}[height=5cm]{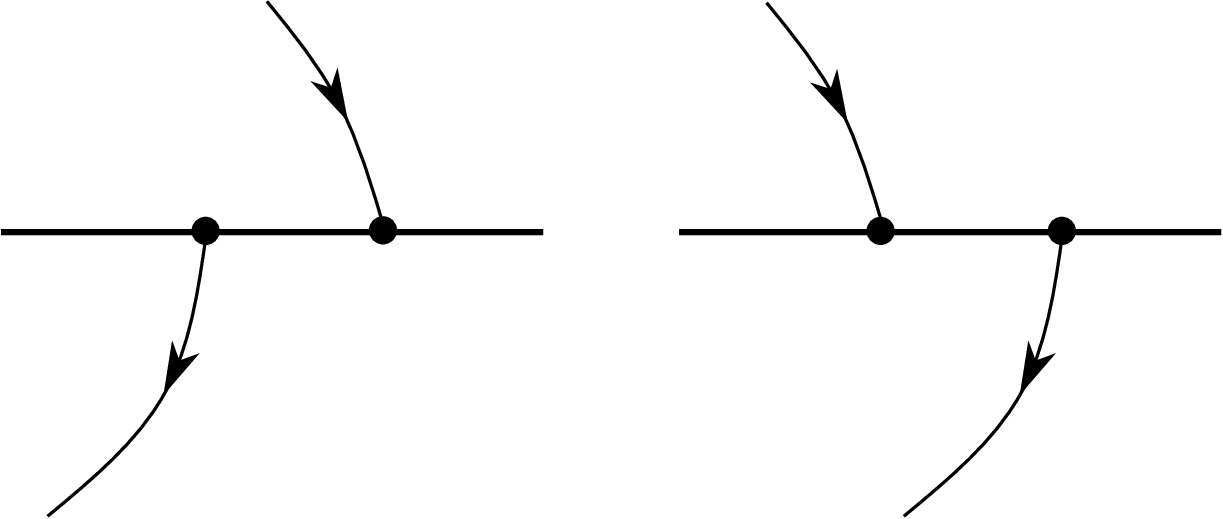} 
			\put(97,24){$\Sigma$}
			\put(42,24){$\Sigma$}
			\put(73,25){$x_0^u$}
			\put(32,25){$x_0^u$}
			\put(87,25){$x_0^s$}
			\put(17,25){$x_0^s$}
			\put(18,18){$d(\mu)>0$}
			\put(72,18){$d(\mu)<0$}
		\end{overpic}
	\end{center}
	\caption{Illustration of $d(\mu)>0$ and $d(\mu)<0$.}\label{Fig5}
\end{figure}
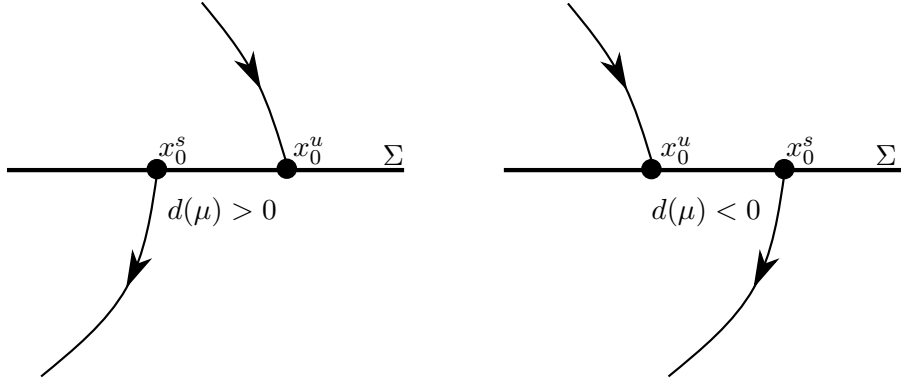
\begin{remark}
	We observe that $L_0$ can intersect $\Sigma$ multiple times. In this case, following Section~\ref{MR}, we write $L_0\cap\Sigma=\{x_0,x_1,\dots,x_n\}$ and let $\gamma_0(t)$ be a parametrization of $L_0$ such that $\gamma_0(t_i)=x_i$, with $t_n<\dots<t_1<t_0=0$. Therefore, applying lemma~\ref{Lemma1} one shall obtain $x_n^u(\mu)$ and then applying the  Implicit Function Theorem multiple times one shall obtain $x_{i}^u(\mu)$ as a function of $x_{i+1}^u(\mu)$, $i\in\{0,\dots,n-1\}$, and thus the displacement function is still well defined at $x_0$.
\end{remark}

Let us define,
	\[x^s_\mu(t)=\phi_1(t,x_0^s(\mu),\mu) \text{ for } t\geqslant0, \quad x^u_\mu(t)=\phi_2(t,x_0^u(\mu),\mu) \text{ for } t\leqslant 0,\]
new parametrizations of $L^s_0(\mu)$ and $L^u_0(\mu)$, respectively. In the following lemma we will denote by $X_i$ some $C^\infty$-extension of $X_i$ at some neighborhood of $\overline{A_i}$, $i\in\{1,2\}$, and thus $x^s_\mu(t)$ and $x^u_\mu(t)$ are well defined for $|t|$ small enough. 

\begin{lemma}\label{Lemma2}
	For any $\mu^*\in\Lambda$ and any $i\in\{1,\dots,n\}$ the maps
	\[\frac{\partial x^s_{\mu^*}}{\partial \mu_i}(t), \quad  \frac{\partial x^u_{\mu^*}}{\partial \mu_i}(t),\]
	are bounded as $t\to+\infty$ and $t\to-\infty$, respectively.
\end{lemma}
\begin{proof} Let us consider a small perturbation of the parameter in the form,
\begin{equation}\label{20}
	\mu=\mu^*+\varepsilon e_i,
\end{equation}
where $e_i$ is the \emph{ith} vector of the canonical base of $\mathbb{R}^r$. The corresponding perturbation of the singularity $p_2(\mu^*)$ takes the form,
	\[p_2(\mu)=p_2(\mu^*)+\varepsilon y_0+o(\varepsilon).\]
Knowing that $X_2(p_2(\mu),\mu)=0$ for any $\varepsilon$ it follows that,
	\[0=\frac{\partial X_2}{\partial \varepsilon}(p_2(\mu),\mu)=DX_2(p_2(\mu),\mu)[y_0+o(\varepsilon)]+\frac{\partial X_2}{\partial \mu}(p_2(\mu),\mu)e_i,\]
and thus applying $\varepsilon\to0$ we obtain, 
	\[y_0=-F_0^{-1}G_0e_i,\]
where $F_0=DX_2(p_2(\mu^*),\mu^*)$ and $G_0=\frac{\partial X_2}{\partial \mu}(p_2(\mu^*),\mu^*)$ (observe that $F_0$ is reversible because $p_2$ is a hyperbolic saddle). Hence,
	\[\frac{\partial p_2}{\partial \mu_i}(\mu^*)=-F_0^{-1}G_0e_i.\]
Therefore, it follows from the $C^\infty$-differentiability of the flow near $p_2(\mu)$ that,
	\[\lim\limits_{t\to-\infty}\frac{\partial x_{\mu^*}^u}{\partial \mu_i}(t)=\frac{\partial p_2}{\partial \mu_i}(\mu^*)=-F_0^{-1}G_0e_i\]
and thus we have the proof for $x_{\mu^*}^u$. The proof for $x_{\mu^*}^s$ is similar. \end{proof}

Let $\theta_i\in(-\pi,\pi)$ be the angle between $X_i(x_0)$ and $u_0$, $i\in\{1,2\}$. See Figure~\ref{Fig6}.
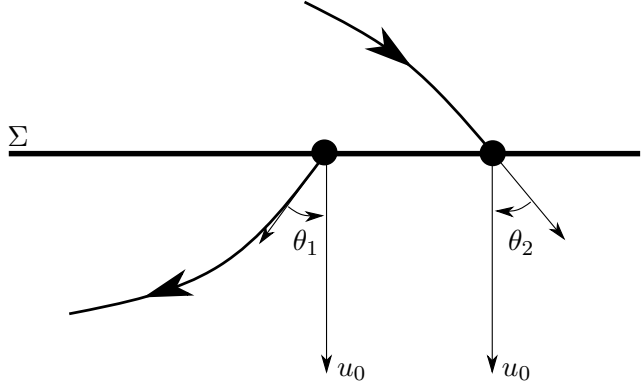
\begin{figure}[ht]
	\begin{center}
		\begin{overpic}[height=5cm]{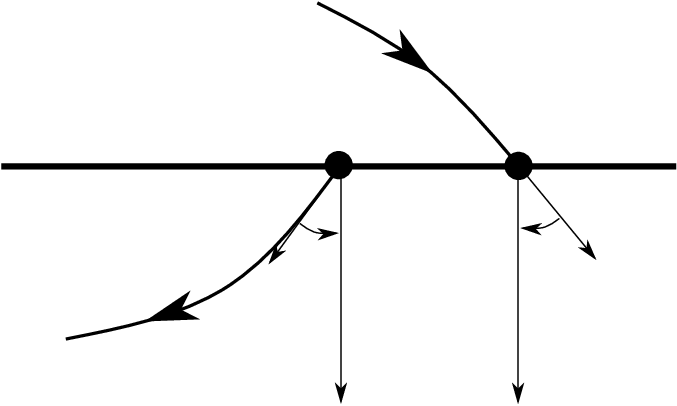} 
			\put(45,20){$\theta_1$}
			\put(79,20){$\theta_2$}
			\put(0,36.5){$\Sigma$}
			\put(52,0){$u_0$}
			\put(78,0){$u_0$}
		\end{overpic}
	\end{center}
	\caption{Illustration of $\theta_1>0$ and $\theta_2<0$.}\label{Fig6}
\end{figure}
For $i\in\{1,2\}$ we denote by $M_i$ the rotation matrix of angle $\theta_i$, i.e.
	\[M_i=\left(\begin{array}{cc} \cos\theta_i & -\sin\theta_i \\ \sin\theta_i & \cos\theta_i\end{array}\right).\]
Following Perko \cite{PerkoHeteroclinic}, we define
	\[\begin{array}{r}
		n^u(t,\mu)=\omega_0[x_\mu^u(t)-x_0]\land u_0, \quad
		n^s(t,\mu)=\omega_0[x_\mu^s(t)-x_0]\land u_0.
	\end{array}\]
It then follows from Definition~\ref{DefDisplacement1} that,
	\[d(\mu)=n^u(0,\mu)-n^s(0,\mu),\]
and thus,
\begin{equation}\label{4}
	\frac{\partial d}{\partial \mu_j}(\mu_0)=\frac{\partial n^u}{\partial \mu_j}(0,\mu_0)-\frac{\partial n^s}{\partial \mu_j}(0,\mu_0).
\end{equation}
Therefore, to understand the displacement function $d(\mu)$, it is enough to understand $n^u$ and $n^s$. Let $X_i=(P_i,Q_i)$, $i\in\{1,2\}$. Knowing that $\gamma_0$ is a parametrization of $L_0$ such that $\gamma_0(0)=x_0$, let $L_0^+=\{\gamma_0(t):t>0\}\subset A_1$ and,
\begin{equation}\label{24}
	I_j^+=\int_{L_0^+}e^{D_1(t)}\left[(M_1X_1)\land\frac{\partial X_1}{\partial \mu_j}(\gamma_0(t),\mu_0)-\sin\theta_1R_{1,j}(\gamma_0(t),\mu_0)\right]dt,
\end{equation}
where,
\[D_i(t)= -\int_{0}^{t}\textnormal{div}X_i(\gamma_0(s),\mu_0)ds,\]
and,
	\[R_{i,j} = \dfrac{\partial P_i}{\partial \mu_j}\left[\left(\dfrac{\partial Q_i}{\partial x_1}+\dfrac{\partial P_i}{\partial x_2}\right)Q_i+\left(\dfrac{\partial P_i}{\partial x_1}-\dfrac{\partial Q_i}{\partial x_2}\right)P_i\right] + \dfrac{\partial Q_i}{\partial \mu_j}\left[\left(\dfrac{\partial P_i}{\partial x_2}+\dfrac{\partial Q_i}{\partial x_1}\right)P_i+\left(\dfrac{\partial Q_i}{\partial x_2}-\dfrac{\partial P_i}{\partial x_1}\right)Q_i\right],\]
$i\in\{1,2\}$, $j\in\{1,\dots,r\}$. 

\begin{proposition}\label{MainProp1}
	For any $j\in\{1,\dots,r\}$ it follows that,
	\[\frac{\partial n^s}{\partial \mu_j}(0,\mu_0)=\frac{\omega_0}{||X_1(x_0,\mu_0)||} I_j^+.\]
\end{proposition}

\begin{proof} From now on in this proof we will denote $X_1$ some $C^\infty$-extension of $X_1$ at some neighborhood of $\overline{A_1}$ and thus $x_\mu^s(t)$ is well define for $|t|$ small enough. Let $j\in\{1,\dots,r\}$. Defining,
	\[\xi(t,\mu)=\frac{\partial x_\mu^s}{\partial \mu_j}(t),\]
it then follows that,
\begin{equation}\label{17}
	\dot \xi(t,\mu) = \dfrac{\partial \dot x_\mu^s}{\partial \mu_j}(t) = \dfrac{\partial }{\partial \mu_j}\left(X_1(x_\mu^s(t),\mu)\right) = DX_1(x_\mu^s(t),\mu)\xi(t,\mu)+\dfrac{\partial X_1}{\partial \mu_j}(x_\mu^s(t),\mu).
\end{equation}
Let $(s,n)=(s(t,\mu),n(t,\mu))$ be the coordinate system with origin at $x_\mu^s(t)$ and such that the angle between $X_1(x_\mu^s(t),\mu)$ and $s$ equals $\theta_1$, and $n$ is orthogonal to $s$, pointing outwards in relation to $G$. See Figure~\ref{Fig7}.
	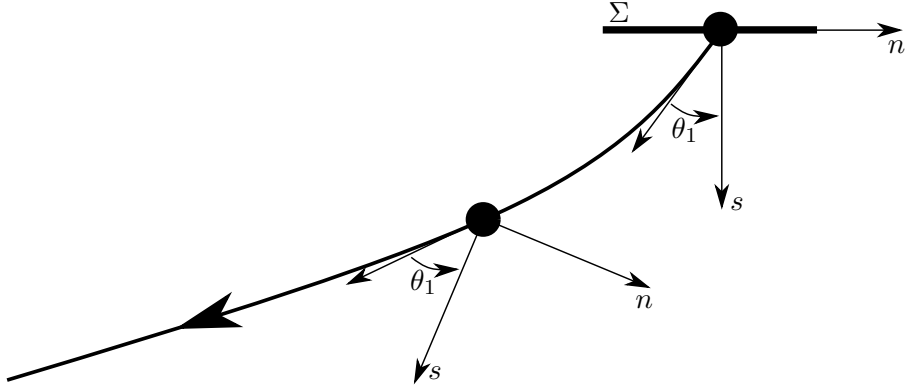
\begin{figure}[ht]
		\begin{center}
			\begin{overpic}[height=5cm]{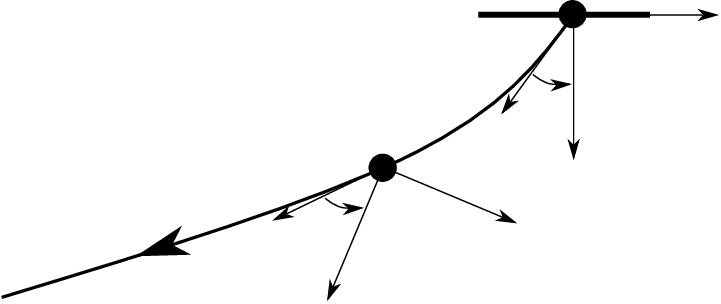}
				\put(98,37.5){$n$}
				\put(70,9){$n$}
				\put(80.5,20){$s$}
				\put(47,1){$s$}
				\put(74,28){$\theta_1$}
				\put(45,11){$\theta_1$}
				\put(67,41){$\Sigma$}
			\end{overpic}
		\end{center}
		\caption{Illustration of $(s,n)$ along $x^s_\mu(t)$.}\label{Fig7}
	\end{figure}
Write $\xi(t,\mu)=\xi_s(t,\mu)s+\xi_n(t,\mu)n$ in function of this new coodinate system and observe that $\xi_n$ (i.e. the component of $\xi$ in the direction of $n$) is given by
	\[\xi_n(t,\mu)=\omega_0\frac{\rho(t,\mu)}{||X_1(x_\mu^s(t,\mu)||},\]
where $\rho(t,\mu)=\xi\land M_1X_1(x_\mu^s(t),\mu)$. Since $n(0,\mu)$ is precisely equal to the normal direction of $u_0$, it follows that,
\begin{equation}\label{5}
	\frac{\partial n^s}{\partial \mu_j}(0,\mu)=\omega_0\frac{\rho(0,\mu)}{||X_1(x_\mu^s(0),\mu)||}.
\end{equation}
Denoting $M_1X_1=(P_0,Q_0)$, $\xi=(\xi_1,\xi_2)$ and, 
\begin{equation}\label{8}
	\rho(t,\mu)=\xi\land M_1X_1(x_\mu^s(t),\mu),
\end{equation}
we conclude that,
	\[\rho=\xi_1 Q_0-P_0\xi_2,\]
where, 
	\[P_0=P_1\cos\theta_1-Q_1\sin\theta_1, \quad Q_0=Q_1\cos\theta_1+P_1\sin\theta_1.\]
Hence,
\begin{equation}\label{10}
	\dot \rho=\dot\xi_1 Q_0-\dot\xi_2 P_0+\xi_1 \dot Q_0-\xi_2 \dot P_0.
\end{equation}
Knowing that,
	\[\dot P_1=\ddot x_1 = \dfrac{\partial P_1}{\partial x_1}P_1+\dfrac{\partial P_1}{\partial x_2}Q_1, \quad \dot Q_1=\ddot x_2 = \dfrac{\partial Q_1}{\partial x_1}P_1+\dfrac{\partial Q_1}{\partial x_2}Q_1,\]
we conclude,
\begin{equation}\label{6}
	\begin{array}{l}
		\dot P_0=\dfrac{\partial P_1}{\partial x_1}P_1\cos\theta_1+\dfrac{\partial P_1}{\partial x_2}Q_1\cos\theta_1-\dfrac{\partial Q_1}{\partial x_1}P_1\sin\theta_1-\dfrac{\partial Q_1}{\partial x_2}Q_1\sin\theta_1, \vspace{0.2cm} \\
		\dot Q_0= \dfrac{\partial Q_1}{\partial x_1}P_1\cos\theta_1+\dfrac{\partial Q_1}{\partial x_2}Q_1\cos\theta_1+\dfrac{\partial P_1}{\partial x_1}P_1\sin\theta_1+\dfrac{\partial P_1}{\partial x_2}Q_1\sin\theta_1.
	\end{array}
\end{equation}
Replacing \eqref{17} and \eqref{6} in \eqref{10} one can conclude, 
\begin{equation}\label{7}
	\dot \rho=\textnormal{div}X_1\rho-M_1X_1\land\dfrac{\partial X_1}{\partial \mu_j}+\sin\theta_1 R_{1,j}.
\end{equation}
Solving \eqref{7} we obtain,
\begin{equation}\label{9}
	\left.\rho(t,\mu)e^{D_1(t)}\right|_{t_0}^{t_1} = \int_{t_0}^{t_1}e^{D_1(t)}\left[\sin\theta_1R_{1,j}(x^s_\mu,\mu)-M_1X_1\land\dfrac{\partial X_1}{\partial \mu_j}(x_\mu^s(t),\mu)\right]dt.
\end{equation}
Observe that $X_1(x_\mu^s(t),\mu)\to0$ as $t\to+\infty$. Therefore, it follows from \eqref{5} and \eqref{8} that $\rho(t,\mu)\to0$ as $t\to+\infty$ (since from lemma~\ref{Lemma2} we know that $\frac{\partial n^s}{\partial \mu_j}$ is bounded). Thus, if we take $t_0=0$ and let $t_1\to+\infty$ in \eqref{9}, then it follows that,
	\[\rho(0,\mu) = \int_{0}^{+\infty}e^{D_1(t)}\left[M_1X_1\land\dfrac{\partial X_1}{\partial \mu_j}(x_\mu^s(t),\mu)-\sin\theta_1R_{1,j}(t,\mu)\right]dt,\]
and thus it follows from \eqref{5} and \eqref{8} we have that,
	\[\frac{\partial n^s}{\partial \mu_j}(0,\mu_0)=\frac{\omega_0}{||X_1(x_0,\mu_0)||} I_j^+.\] \end{proof}

\begin{remark}
	Observe that even if $L_0$ intersects $\Sigma$ in multiple points, $L_0^+$ was defined in such a way that there is no discontinuities on it. Moreover, if $L_0\cap \Sigma=\{x_0\}$, then $L_0^-=\{\gamma_0(t):t<0\}$ also has no discontinuities and thus, as in Proposition~\ref{MainProp1}, one can prove that,
		\[\frac{\partial n^u}{\partial \mu_j}(t,\mu)=\omega_0\frac{\overline{\rho}(t,\mu)}{||X_2(x_\mu^u(t),\mu)||},\]
	with $\overline{\rho}$ satisfying \eqref{7}, but with $x_\mu^u$ instead of $x_\mu^s$ and $X_2$ instead of $X_1$. Furthermore we have $\overline{\rho}(t,\mu)\to0$ as $t\to-\infty$ and thus by setting $t_1=0$ and letting $t_0\to-\infty$ we obtain,
		\[\overline{\rho}(0,\mu) = -\int_{0}^{+\infty}e^{D_2(t)}\left[M_2X_2\land\dfrac{\partial X_2}{\partial \mu_j}(x_\mu^u(t),\mu)-\sin\theta_2R_{2,j}(t,\mu)\right]dt.\] 
	Hence, in the simple case where $L_0$ intersects $\Sigma$ in an unique point $x_0$, it follows from \eqref{4} and from Proposition~\ref{MainProp1} that,
		\[\dfrac{\partial d}{\partial \mu_j}(\mu_0)=-\omega_0\left(\frac{1}{||X_2(x_0,\mu_0)||}I_j^-+\frac{1}{||X_1(x_0,\mu_0)||}I_j^+\right),\]
	with,
	\[\begin{array}{rl}
		I_j^+&=\displaystyle\int_{L_0^+}e^{D_1(t)}\left[(M_1X_1)\land\dfrac{\partial X_1}{\partial \mu_j}(\gamma_0(t),\mu_0)-\sin\theta_1R_{1,j}(\gamma_0(t),\mu_0)\right]dt, \vspace{0.2cm} \\ 
		I_j^-&=\displaystyle\int_{L_0^-}e^{D_2(t)}\left[(M_2X_2)\land\dfrac{\partial X_2}{\partial \mu_j}(\gamma_0(t),\mu_0)-\sin\theta_2R_{2,j}(\gamma_0(t),\mu_0)\right]dt.
	\end{array}\]
\end{remark}

\begin{remark}
	If instead of a non-smooth vector field we suppose that $Z=X$ is smooth, then we can assume $X_1=X_2$ and take $u_0=\frac{X(x_0,\mu_0)}{||X(x_0,\mu_0)||}$. In this case we would have $\theta_1=\theta_2=0$ and therefore conclude that,
		\[\frac{\partial d}{\partial \mu_j}(\mu_0)=-\frac{\omega_0}{||X(x_0,\mu_0)||}\int_{-\infty}^{+\infty}e^{-\int_{0}^{t}\textnormal{div}X(\gamma_0(s),\mu_0)ds}\left[X\land\frac{\partial X}{\partial \mu_j}(\gamma_0(t),\mu_0)\right]dt,\]
	as in the works of Perko, Holmes and Guckenheimer \cites{PerkoHeteroclinic,HolmesHeteroclinic,GuckHolmes}.
\end{remark}

\begin{remark}
	Within this section, the hypothesis of a polycycle is not necessary. In fact, if we assume only a heteroclinic connection between saddles, then we
	can define the displacement function as
	\[d(\mu)=[x_0^u(\mu)-x_0^s(\mu)]\land u_0,\]
	and therefore obtain,
	\[\dfrac{\partial d}{\partial \mu_j}(\mu_0)=-\left(\frac{1}{||X_2(x_0,\mu_0)||}I_j^-+\frac{1}{||X_1(x_0,\mu_0)||}I_j^+\right).\]
	It is only necessary to pay attention at which direction we have $d(\mu)>0$ or $d(\mu)<0$. 
\end{remark}

Let us now study the case where at least one of the endpoints of the heteroclinic connection is a tangential singularity. In the case of the hyperbolic saddle, we use the Center-Stable Manifold Theorem \cite{Kel1967} to take a point $x_1^s(\mu)$ within the stable manifold of the hyperbolic saddle. Then, we define,
\begin{equation}\label{19}
	x^s(t,\mu)=\phi_1(t,x_1^s(\mu),\mu),
\end{equation}
where $\phi_1$ is the flow of $X_1$, the component of $Z$ which contains the hyperbolic saddle. Then, we use the Implicit Function Theorem at lemma~\ref{Lemma1} to obtain a smooth function $\tau^s(\mu)$ such that,
\begin{equation}\label{21}
	x_0^s(\mu)=x^s(\tau^s(\mu),\mu)\in\Sigma,
\end{equation}
for every $\mu\in\Lambda$, where $\Lambda$ is a small enough neighborhood of $\mu_0$. Therefore, in the case of the tangential singularity, we define
\begin{equation}\label{22}
	x_1^s(\mu)=T^s(h_\mu(0),\mu),
\end{equation}
where $T^s$ and $h_\mu$ are given by \eqref{26}. However, in the case of a hyperbolic saddle the definition of $x^s(t,\mu)$ given by \eqref{19} works essentially because the hyperbolic saddle is \emph{structurally stable}. But this may not be the case of a tangential singularity. Moreover, the parameters at \eqref{26}, related with the transitions maps near tangential singularities, depend \emph{continuously} on the parameter $\mu$. Hence, we cannot take its derivative with respect to $\mu$. To avoid these problems, it is sufficient to assume that $Z$ is constant in a neighborhood of the tangential singularity.

\begin{remark}\label{StrongHypothesis}
	From now on, given a tangential singularity $p$ of the planar non-smooth vector field $Z$, we suppose that $Z$ is constant in a neighborhood $B$ of $p$.
\end{remark}

In this case, let $x_1^s(\mu)\in B$ be given by \eqref{22}. It follows from Remark~\ref{StrongHypothesis} that $x_1^s(\mu)=x_1^s$ is constant. Now, similarly to the case of the hyperbolic saddle, let $x^s(t,\mu)$ be given by \eqref{19} (i.e. $x^s(t,\mu)$ is the parametrization of the regular orbit $L_0^s(\mu)$). Let $x_0=x^s(t_0^s,\mu_0)$ be the intersection of $L_0(\mu_0)$ and $\Sigma$. See Figure~\ref{Fig12}. 
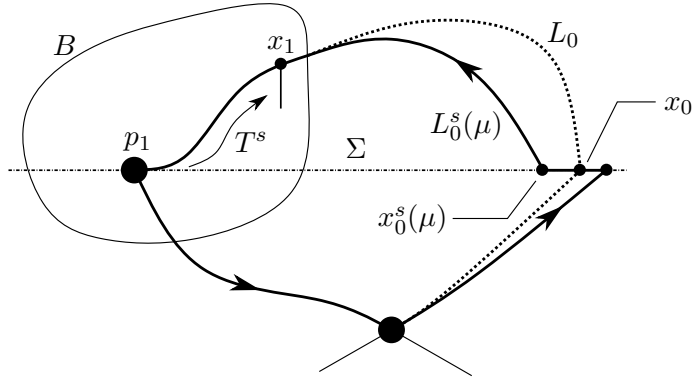
\begin{figure}[ht]
	\begin{center}
		\begin{overpic}[height=5cm]{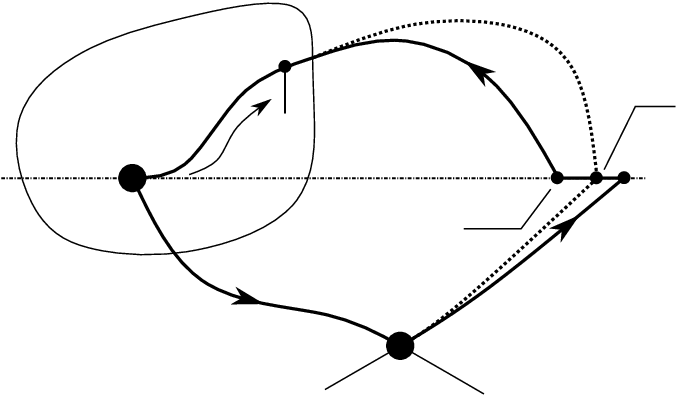} 
			\put(7,50){$B$}
			\put(18,36){$p_1$}
			\put(35,35){$T^s$}
			\put(40,51){$x_1$}
			\put(83,52){$L_0$}
			\put(65,38){$L_0^s(\mu)$}
			\put(52,34){$\Sigma$}
			\put(101,41.5){$x_0$}
			\put(57,23.5){$x_0^s(\mu)$}
		\end{overpic}
	\end{center}
	\caption{Illustration of $x_0^s(\mu)$ and $x_0^u(\mu)$ in the case of a tangential singularity.}\label{Fig12}
\end{figure}
Similarly to lemma~\ref{Lemma1}, it follows from the Implicit Function Theorem that there exist a $C^\infty$-function $\tau^s(\mu)$ such that $\tau^s(\mu)\to t_0^s$, as $\mu\to\mu_0$, and such that $x_0^s(\mu)=x^s(\tau^s(\mu),\mu)\in\Sigma$, for all $\mu\in\Lambda$. Similarly, one can define $x_1^u(\mu)$ and then obtain $\tau^u(\mu)$ and $x_0^u(\mu)$. Therefore, we have obtained an analogous version of lemma~\ref{Lemma1} for the case of tangential singularities. Since $Z$ is constant in a neighborhood of $p_1$, it follows that the partial derivatives of $x^s$ in relation to $\mu$ are zero near $p_1$. Hence, lemma~\ref{Lemma2} is also clear in this case. Although the assumption made at Remark~\ref{StrongHypothesis} is strong, we observe that to prove Theorem~\ref{Main4}, we will make use of \emph{bump-functions} to construct perturbations that does not affect any tangential singularity. We now prove the similar version of Proposition~\ref{MainProp1}.

\begin{proposition}\label{MainProp2}
	Let $p_1$ be a tangential singularity satisfying Remark~\ref{StrongHypothesis}. Then for any $j\in\{1,\dots,r\}$ it follows that,
		\[\frac{\partial n^s}{\partial \mu_j}(0,\mu_0)=\frac{\omega_0}{||X_1(x_0,\mu_0)||}(I_j^+ + H_j^+),\]
	where $I_j^+$ is given by \eqref{24}.
		\[H_j^+=e^{D_1(t_1)}\frac{\partial \gamma_0}{\partial \mu_j}(t_1)\land M_1 X_1(p_1,\mu_0),\]
	and $\gamma_0$ is a parametrization of $L_0^+$ such that $\gamma_0(0)=x_0$ and $\gamma_0(t_1)=p_1$.
\end{proposition}

\begin{proof} Let $\gamma_0$ be a parametrization of $L_0$ such that $\gamma_0(0)=x_0$ and $\gamma_0(t_1)=p_1$. It follows from Proposition~\ref{MainProp1} that,
	\[\frac{\partial n^s}{\partial \mu_j}(0,\mu)=\omega_0\frac{\rho(0,\mu)}{||X_1(x_\mu^s(0),\mu)||},\]
with,
	\[\rho(t,\mu)=\xi\land M_1X_1(x_\mu^s(t),\mu),\]
satisfying,
\begin{equation}\label{11}
	\left.\rho(t,\mu)e^{D_1(t)}\right|_{t_0}^{t_1} = \int_{t_0}^{t_1}e^{D_1(t)}\left[\sin\theta_1R_{1,j}(x^s_\mu,\mu)-M_1X_1\land\dfrac{\partial X_1}{\partial \mu_j}(x_\mu^s(t),\mu)\right]dt.
\end{equation}
Observe that $t_1>0$ and thus we can define $L_0^+=\{\gamma_0(t):0<t<t_1\}$. Then, it follows from \eqref{11} that,
	\[\rho(0,\mu_0)=I_j^++e^{D_1(t_1)}\rho(t_1,\mu_0).\]
Since $Z$ is constant in a neighborhood of $p_1$, it follows that $\xi(t_1,\mu)=0$ and thus $\rho(t_1,\mu_0)=0$. This finishes the prove. \end{proof}

In the following proposition we use the Poincar\'e-Bendixson theory for non-smooth vector fields (see Buzzi et al \cite{PB}) and the displacement maps to prove, under some conditions, the bifurcation of limit cycles. Such result will be used in a induction argumentation in the proof of Theorem~\ref{Main4}. 

\begin{proposition}\label{Prop5}
	Let $Z$ and $\Gamma^n$ be as in Section~\ref{MR}, with the tangential singularities satisfying Remark~\ref{StrongHypothesis}, and $d_i:\Lambda\to\mathbb{R}$, $i\in\{1,\dots,n\}$, be the displacement maps defined at the regular orbits of $\Gamma^n$. Let $\sigma_0\in\{-1,1\}$ be a constant such that $\sigma_0=1$ (resp. $\sigma_0=-1$) if the Poincar\'e map is defined in the bounded (resp. unbounded) region of $\Gamma^n$. Then following statements holds.
	\begin{enumerate}[label=(\alph*)]
		\item If $r(\Gamma^n)>1$ and $\mu\in\Lambda$ is such that $\sigma_0d_1(\mu)\leqslant0,\dots, \sigma_0d_n(\mu)\leqslant0$ with $\sigma_0d_i(\mu)<0$ for some $i\in\{1,\dots,n\}$, then at least one stable limit cycle $\Gamma$ bifurcates from $\Gamma^n$.
		\item If $r(\Gamma^n)<1$ and $\mu\in\Lambda$ is such that $\sigma_0d_1(\mu)\geqslant0,\dots,\sigma_0d_n(\mu)\geqslant0$ with $\sigma_0d_i(\mu)>0$ for some $i\in\{1,\dots,n\}$, then at least one unstable limit cycle $\Gamma$ bifurcates from $\Gamma^n$.
	\end{enumerate} 
\end{proposition}

\begin{proof} For the simplicity we will use the same polycycle $\Gamma$ used in the proof of Theorem~\ref{Main1}. Let $x_{i,0}\in L_i$ be as in Section~\ref{MR} and $l_i$ be transversal sections of $L_i$ through $x_{i,0}$, $i\in\{1,2,3\}$. Let $R_i:l_i\times\Lambda\to l_{i-1}$ be functions given by the compositions of the functions used in the proof of Theorem~\ref{Main1}, $i\in\{1,2,3\}$. See Figure~\ref{Fig9}. 
\begin{figure}[ht]
	\begin{center}
		\begin{minipage}{6cm}
			\begin{center}
				\begin{overpic}[width=5.5cm]{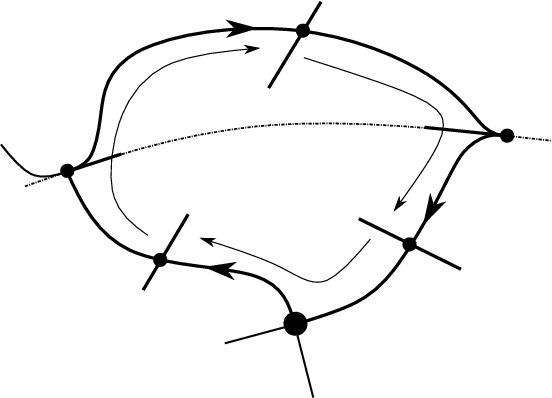} 
					\put(30,53){$R_1$}
					\put(65,40){$R_2$}
					\put(45,28){$R_3$}
				\end{overpic}
					
					$d_i(\mu)=0$.
			\end{center}
		\end{minipage}
		\begin{minipage}{6cm}
			\begin{center}
				\begin{overpic}[width=5.5cm]{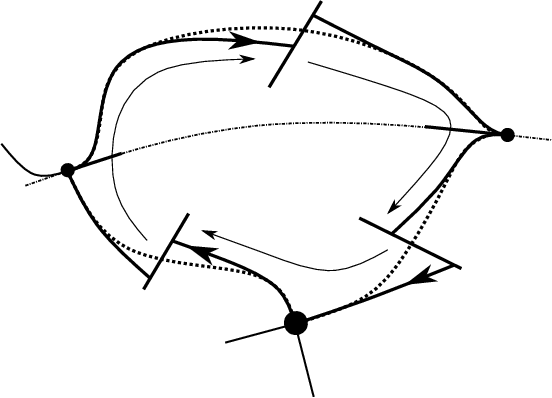} 
				\end{overpic}
					
					$d_i(\mu)<0$.
			\end{center}
		\end{minipage}
	\end{center}
\caption{Observe that if $d_i(\mu)>0$, then the composition may not be well defined.}\label{Fig9}
\end{figure}
Hence, if $d_i(\mu)\leqslant0$, $i\in\{1,2,3\}$, then the Poincar\'e map $P:l_1\times\Lambda\to l_1$ can be written as $P(x,\mu)=R_3(R_2(R_1(x,\mu),\mu),\mu)$. We observe that $P$ is $C^\infty$ in $x$, continuous in $\mu$, and it follows from the proof of Theorem~\ref{Main1} that $P(\cdot,\mu_0)$ is non-flat. It follows from Theorem~\ref{Main1} that there is an open ring $A$ in the bounded region delimited by $\Gamma$, such that the orbit $\Gamma$ through any point $q\in A_0$ spiral towards $\Gamma^3$ as $t\to+\infty$. Let $p$ be the interception of $\Gamma$ and $l_1$, $q_0\in A_0\cap l_1$, $\xi$ a coordinate system along $l_1$ such that $\xi=0$ at $p$ and $\xi>0$ at $q_0$ and let we identify this coordinate system with $\mathbb{R}_+$. Observe that $P(q_0,\mu_0)<q_0$ and thus by continuity $P(q_0,\mu)<q_0$ for any $\mu\in\Lambda$. See Figure~\ref{Fig10}.
\begin{figure}[ht]
	\begin{center}
		\begin{overpic}[width=8cm]{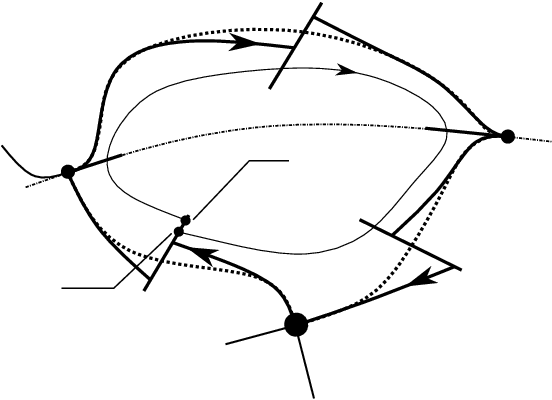} 
			\put(53,42){$q_0$}
			\put(-7,19){$P(q_0,\mu)$}
		\end{overpic}
	\end{center}
	\caption{Illustration of the application of the Poincar\'e-Bendixson theory.}\label{Fig10}
\end{figure}
Therefore, it follows from the Poincar\'e-Bendixson theory and from the non-flatness of $P$ that at least one stable limit cycle $\Gamma_0$ bifurcates from $\Gamma^3$. Statement $(b)$ can be prove by time reversing. \end{proof}

\section{The further displacement map}\label{sec5}

Let $Z$ and $\Gamma^n$ be as in Section~\ref{MR}, with the tangential singularities satisfying Remark~\ref{StrongHypothesis}. Let $L_i^u(\mu)$ and $L_i^s(\mu)$ be the perturbations of $L_i$ such that $\alpha(L_i^u)=p_{i+1}$ and $\omega(L_i^s)=p_i$, $i\in\{1,\dots,n\}$, with each index being modulo $n$. Following the work of Han et al \cite{HanWuBi}, let $C_i=x_{i,0}$. If $C_i\notin\Sigma$, then let $v_i$ be the unique unitary vector orthogonal to $Z(C_i,\mu_0)$ and pointing outwards in relation to $\Gamma^n$. See Figure~\ref{FigExtra1}. 
\begin{figure}[ht]
	\begin{center}
		\begin{overpic}[width=8cm]{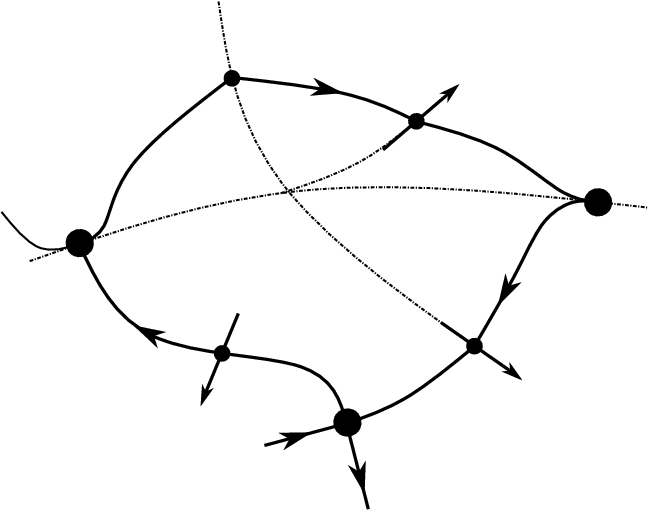} 
			\put(91,51){$p_1$}
			\put(10,45){$p_2$}
			\put(56,11){$p_3$}
			\put(62,54){$C_1$}
			\put(26,68){$x_{1,1}$}
			\put(37,26){$C_2$}
			\put(65,24){$C_3$}
			\put(42,44){$\Sigma$}
			\put(72,64){$l_1$}
			\put(34,15){$l_2$}
			\put(80,22){$l_3$}
		\end{overpic}
	\end{center}
	\caption{An example of the construction of the points $C_i$ and the lines $l_i$.}\label{FigExtra1}
\end{figure}
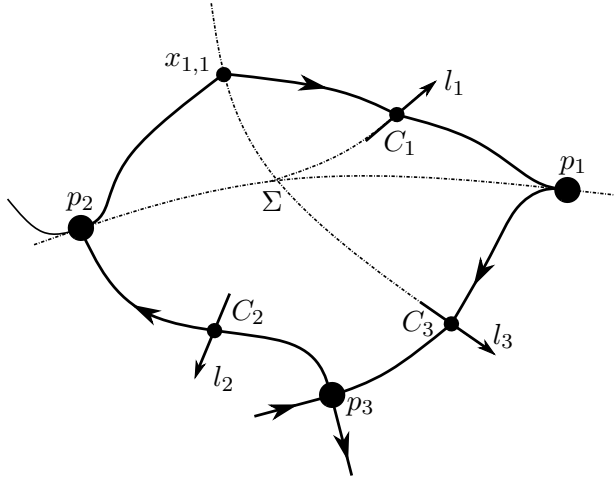
On the other hand, if $C_i\in\Sigma$, then let $v_i$ be the unique unitary vector tangent to $T_{C_i}\Sigma$ and pointing outwards in relation to $\Gamma^n$. In both cases, let $l_i$ be the transversal section normal to $L_i$ at $C_i$. It is clear that any point $B\in l_i$ can be written as $B=C_i+\lambda v_i$, with $\lambda\in\mathbb{R}$. Moreover, let $N_i$ be a small enough neighborhood of $C_i$ and $J_i=N_i\cap\Sigma$. It then follows that any point $B\in J_i$ can be orthogonally projected on the line $l_i:C_i+\lambda v_i$, $\lambda\in\mathbb{R}$, and thus it can be uniquely, and smoothly, identified with $C_i+\lambda_B v_i$, for some $\lambda_B\in\mathbb{R}$. In either case $C_i\in\Sigma$ or $C_i\not\in\Sigma$, observe that if $\lambda>0$, then $B$ is outside $\Gamma^n$ and if $\lambda<0$, then $B$ is inside $\Gamma^n$. For each $i\in\{1,\dots,n\}$ we define, 
\begin{equation}\label{12}
	B_i^u=L_i^u\cap l_i=C_i+b_i^u(\mu)v_i, \quad B_i^s=L_i^s\cap l_i=C_i+b_i^s(\mu)v_i.
\end{equation}
Therefore, it follows from Section~\ref{sec4} that,
	\[d_i(\mu)=b_i^u(\mu)-b_i^s(\mu),\]
$i\in\{1,\dots,n\}$. Let $r_i=\frac{|\nu_i(\mu_0)|}{\lambda_i(\mu_0)}$ if $p_i$ is a hyperbolic saddle or $r_i=\frac{n_{i,u}}{n_{i,s}}$ if $p_i$ is a tangential singularity, $i\in\{1,\dots,n\}$. If $r_i>1$ and $d_i(\mu)<0$, then following \cite{HanWuBi}, we observe that,
	\[B_{i-1}^*=L_i^u\cap l_{i-1}=C_{i-1}+b_{i-1}^*(\mu)v_{i-1},\]
is well defined and thus we define the \emph{further displacement map} as,
	\[d_{i-1}^*(\mu)=b_{i-1}^*(\mu)-b_{i-1}^s(\mu).\]
See Figure~\ref{Fig11}. 
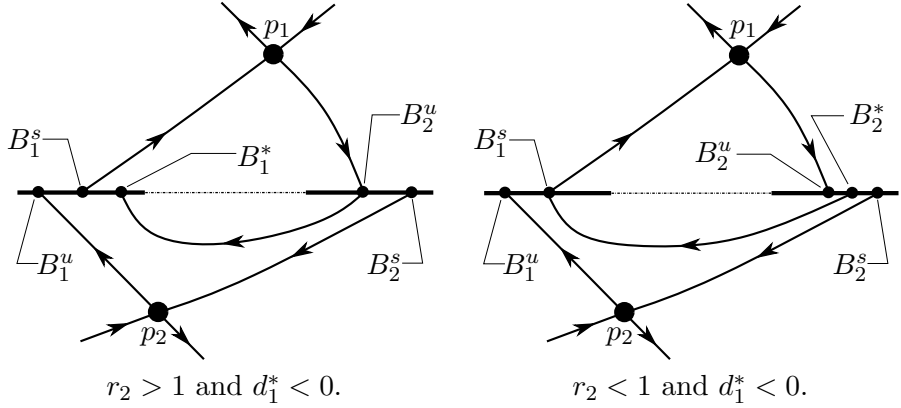
\begin{figure}[ht]
	\begin{center}
		\begin{minipage}{6cm}
			\begin{center}
				\begin{overpic}[width=5.5cm]{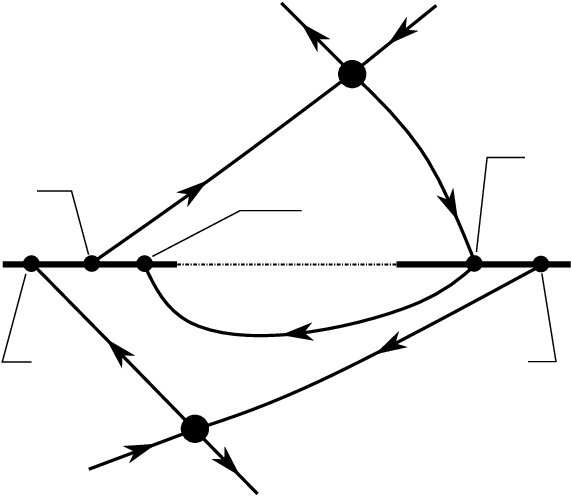}
					\put(53,47){$B_1^*$}
					\put(-2,51){$B_1^s$}
					\put(5,21){$B_1^u$}
					\put(92,57){$B_2^u$}
					\put(84,21){$B_2^s$}
					\put(59.5,79){$p_1$}
					\put(30,5){$p_2$}
				\end{overpic}
				
				$r_2>1 \text{ and } d_1^*<0.$
			\end{center}
		\end{minipage}
		\begin{minipage}{6cm}
			\begin{center}
				\begin{overpic}[width=5.5cm]{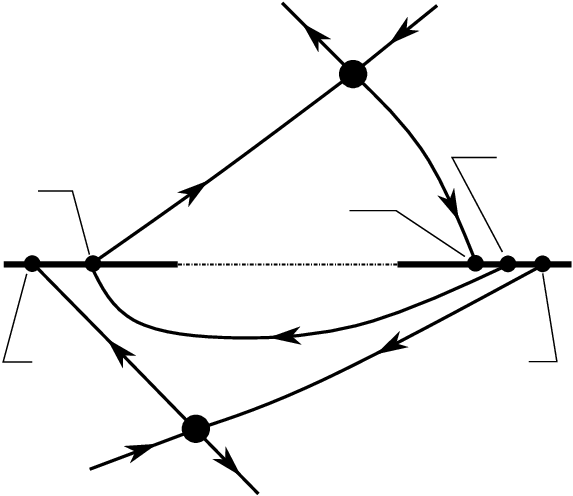} 
					\put(88,57){$B_2^*$}
					\put(-2,51){$B_1^s$}
					\put(5,21){$B_1^u$}
					\put(51,48){$B_2^u$}
					\put(84,21){$B_2^s$}
					\put(59.5,79){$p_1$}
					\put(30,5){$p_2$}
				\end{overpic}
				
				$r_2<1 \text{ and } d_1^*<0.$
			\end{center}
		\end{minipage}
	\end{center}
	\caption{Illustration of $d_1^*<0$ for both $r_2>1$ and $r_2<1$.}\label{Fig11}
\end{figure}
On the other hand, if $r_i<1$ and $d_{i-1}(\mu)>0$, then,
\[B_i^*=L_{i-1}^s\cap l_i=C_i+b_i^*(\mu)v_i,\]
is well defined and thus we define the \emph{further displacement map} as,
\[d_{i-1}^*(\mu)=b_i^u(\mu)-b_i^*(\mu).\] 

\begin{proposition}\label{Prop3}
	For $i\in\{1,\dots,n\}$ and $\Lambda\subset\mathbb{R}^r$ small enough we have,
	\[d_{i-1}^*(\mu)=\left\{\begin{array}{ll}
		d_{i-1}(\mu)+O(||\mu-\mu_0||^{r_i}) & \text{if } r_i>1, \vspace{0.2cm} \\
		d_i(\mu)+O(||\mu-\mu_0||^\frac{1}{r_i}) & \text{if } r_i<1.
	\end{array}\right.\]
\end{proposition}

\begin{proof} For simplicity let us assume $i=n$ and $r_n>1$. It follows from the definition of $d_{n-1}^*$ and $d_{n-1}$ that,
\begin{equation}\label{13}
	d_{n-1}^*=(b_{n-1}^*-b_{n-1}^u)+d_{n-1}.
\end{equation}  
Let $B=B_n^s+\lambda v_n\in l_n$, $\lambda<0$, with $|\lambda|$ small enough and observe that the orbit through $B$ will intersect $l_{n-1}$ in a point $C$ which can be written as,
	\[C=B_{n-1}^u+F(\lambda,\mu)v_{n-1}.\]
Therefore, we have a function $F:l_n\to l_{n-1}$ with $F(\lambda,\mu)<0$ for $\lambda<0$, $|\lambda|$ small enough, such that $F(\lambda,\mu)\to0$ as $\lambda\to0$. From \eqref{12} we have,
\begin{equation}\label{14}
	\begin{array}{l}
		B_n^u=(C_n+b_n^sv_n)+(b_n^u-b_n^s)v_n=B_n^s+d_nv_n \vspace{0.2cm} \\
		B_{n-1}^*=(C_{n-1}+b_{n-1}^uv_{n-1})+(b_n^*-b_{n-1}^u)v_{n-1}=B_{n-1}^u+(b_n^*-b_{n-1}^u)v_{n-1}.
	\end{array}
\end{equation}
Since $B_{n-1}^*$ is the intersection of the positive orbit through $B_n^u$ with $l_{n-1}$ it follows from \eqref{14} that,
	\[b_n^*-b_{n-1}^u=F(d_n,\mu).\]
Therefore, it follows from \eqref{13} that,
\begin{equation}\label{15}
	d_{n-1}^*=F(d_n,\mu)+d_{n-1}.
\end{equation}
If $p_i$ is a hyperbolic saddle, then $F$ is, up to the composition of some diffeomorphisms given by the flow of the components of $Z$, the Dulac map $D_i$ defined at Section~\ref{sec2.1}. If $p_i$ is a tangential singularity (we remember that we are under the hypothesis of Remark~\ref{StrongHypothesis}), then $F$ is, up to the composition of some diffeomorphisms given by the flow of the components of $Z$, the composition $T_i^u\circ(T_i^s)^{-1}$ defined at Section~\ref{sec2.3}. In either case, it follows from Section~\ref{sec2} that, 
\begin{equation}\label{16}
	|F(\lambda,\mu)|=|\lambda|^{r_n}(A(\mu)+O(1)),
\end{equation}
with $A(\mu_0)\neq0$. Since $d_n=O(||\mu||)$, it follows from \eqref{16} that,
	\[F(d_n,\mu)=O(||\mu||^{r_n}),\]
and thus from \eqref{15} we have the result. The case $r_n<1$ follows similarly from the fact that the inverse $F^{-1}$ has order $r_n^{-1}$ in $u$. \end{proof}

\begin{corollary}
	For each $i\in\{1,\dots,n\}$ the further displacement map $d_i^*$ is continuous differentiable with the $j$-partial derivative given either by the $j$-partial derivative of $d_i$ or $d_{i+1}$. Furthermore a connection between $p_i$ and $p_{i-2}$ exists if and only if $d_{i-2}^*(\mu)=0$ and $d_{i-1}(\mu)\neq0$.  
\end{corollary}

\section{Proof of Theorem~\ref{Main4}}\label{sec6}

\noindent {\it Proof of Theorem~\ref{Main4}.} Let $Z=(X_1,\dots,X_M;\Sigma)$ and denote $X_i=(P_i,Q_i)$, $i\in\{1,\dots,M\}$. Let $\{p_1,\dots,p_n\}$ be the singularities of $\Gamma^n$ and $L_i$ the regular orbits between them such that $\omega(L_i)=p_i$ and $\alpha(L_i)=p_{i+1}$. If $L_i\cap\Sigma=\emptyset$, then take $x_{i,0}\in L_i$ and $\gamma_i(t)$ a parametrization of $L_i$ such that $\gamma_i(0)=x_{i,0}$. If $L_i\cap\Sigma\neq\emptyset$, then let $L_i\cap\Sigma=\{x_{i,0},\dots,x_{i,n(i)}\}$ and take $\gamma_i(t)$ a parametrization of $L_i$ such that $\gamma_i(t_{i,j})=x_{i,j}$ with $0=t_{i,0}>t_{i,1}>\dots>t_{i,n(i)}$. In either case denote $L_i^+=\{\gamma_i(t):t>0\}$ if $p_i$ is a hyperbolic saddle or $L_i^+=\{\gamma_i(t):0<t<t_i\}$, where $t_i$ is such that $\gamma_i(t_i)=p_i$, if $p_i$ is a $\Sigma$-singularity. Following \cite{HanWuBi}, for each $i\in\{1,\dots,n\}$ let $G_{i,j}$, $j\in\{1,2\}$, be two compact disks small enough such that,
	\begin{enumerate}[label=\arabic*)]
		\item $\Gamma^n\cap G_{i,j}=L_i^+\cap G_{i,j}\neq\emptyset$, $j\in\{1,2\}$;
		\item $G_{i,1}\subset \text{Int}G_{i,2}$;
		\item $G_{i,2}\cap G_{s,2}=\emptyset$ for any $i\neq s$;
		\item $G_{i,j}\cap\Sigma=\emptyset$.
	\end{enumerate}
Let $k_i:\mathbb{R}^2\to[0,1]$ be a $C^\infty$-bump function such that,
	\[k_i(x)=\left\{\begin{array}{l} 0, \quad x\notin G_{i,2}, \\ 1, \quad x\in G_{i,1}. \end{array}\right.\]	
See Figure~\ref{FigExtra2}.
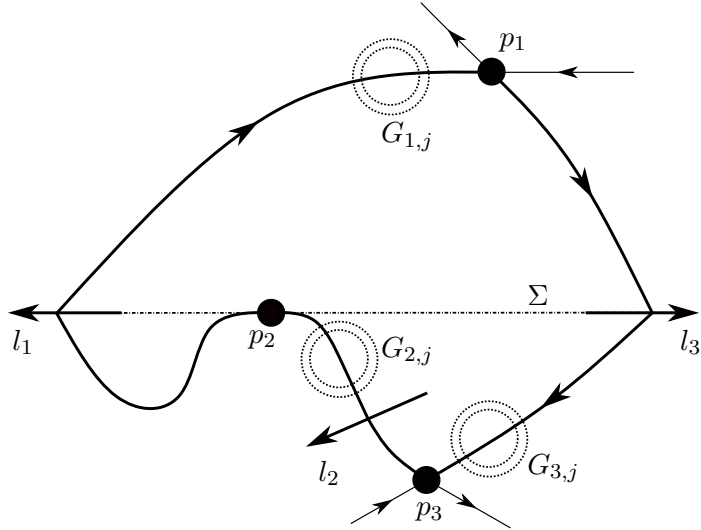
\begin{figure}[ht]
	\begin{center}
		\begin{overpic}[height=7cm]{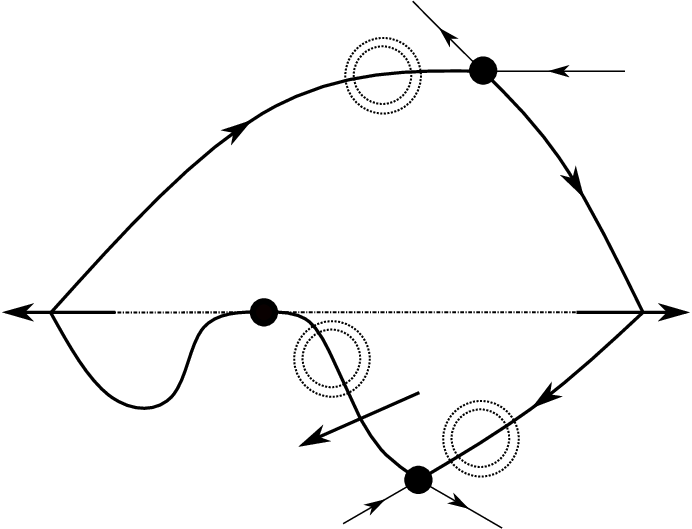}
			\put(71,70){$p_1$}
			\put(35,27){$p_2$}
			\put(59,2){$p_3$}
			\put(75,32.5){$\Sigma$}
			\put(54,56){$G_{1,j}$}
			\put(54,25){$G_{2,j}$}
			\put(75,8){$G_{3,j}$}
			\put(1,26){$l_1$}
			\put(45,7){$l_2$}
			\put(97,26){$l_3$}
		\end{overpic}
	\end{center}
	\caption{Illustration of the sets $G_{i,j}$.}\label{FigExtra2}
\end{figure}
Let $\mu\in\mathbb{R}^n$ and $g_i:\mathbb{R}^2\to\mathbb{R}^2$, $i\in\{1,\dots,n\}$, be maps that we yet have to define. Let also, 
	\[g(x,\mu)=\sum_{i=1}^{n}\mu_ik_i(x)g_i(x),\]
and for now one let us denote $X_i=X_i+g$. Let $\Lambda$ be a small enough neighborhood of the origin of $\mathbb{R}^n$. It follows from Section~\ref{sec4} that each displacement map $d_i:\Lambda\to\mathbb{R}$ controls the bifurcations of $L_i$ near $x_{i,0}$. It follows from Definition~\ref{DefDisplacement1} that,
	\[d_i(\mu)=\omega_0[x_{i,0}^u(\mu)-x_{i,0}^s(\mu)]\land \eta_i,\]
where $\eta_i$ is the analogous of $u_0$ in Figure~\ref{Fig3}. But from the definition of $g$ we have that each $x_{i,0}^u(\mu)$ does not depend on $\mu$ and thus $x_{i,0}^u\equiv x_{i,0}$. Furthermore it follows from the definition of $k_i$ that each singularity $p_i$ of $\Gamma^n$ also does not depend on $\mu$ and thus $\frac{\partial \gamma_i}{\partial \mu_j}(t_i)=0$ for every tangential singularity $p_i$. Therefore, it follows from Propositions~\ref{MainProp1} and \ref{MainProp2} that,
	\[\frac{\partial d_i}{\partial\mu_j}(0)=-\frac{\omega_0}{||X_i(x_0,\mu_0)||}\int_{L_i^+}e^{D_i(t)}\left[(M_iX_i)\land\frac{\partial X_i}{\partial \mu_j}(\gamma_i(t),0)-\sin\theta_iR_{i,j}(\gamma_i(t),0)\right]dt,\]
with,
	\[D_i(t)= -\int_{0}^{t}\textnormal{div}X_i(\gamma_0(s),\mu_0)ds,\]
and,
	\[R_{i,j} = \dfrac{\partial P_i}{\partial \mu_j}\left[\left(\dfrac{\partial Q_i}{\partial x_1}+\dfrac{\partial P_i}{\partial x_2}\right)Q_i+\left(\dfrac{\partial P_i}{\partial x_1}-\dfrac{\partial Q_i}{\partial x_2}\right)P_i\right] + \dfrac{\partial Q_i}{\partial \mu_j}\left[\left(\dfrac{\partial P_i}{\partial x_2}+\dfrac{\partial Q_i}{\partial x_1}\right)P_i+\left(\dfrac{\partial Q_i}{\partial x_2}-\dfrac{\partial P_i}{\partial x_1}\right)Q_i\right],\]
$i$, $j\in\{1,\dots,n\}$. We observe that if $L_i\cap\Sigma=\emptyset$, then $\theta_i=0$. It follows from the definition of the sets $G_{i,j}$ that $\frac{\partial d_i}{\partial \mu_j}(0)=0$ if $i\neq j$. Let $M_iX_i=(\overline{P}_i,\overline{Q}_i)$ and $R_{i,i}=\frac{\partial P_i}{\partial \mu_i}F_{i,1}+\frac{\partial Q_i}{\partial \mu_i}F_{i,2}$, where,
	\[F_{i,1}=\left(\dfrac{\partial Q_i}{\partial x_1}+\dfrac{\partial P_i}{\partial x_2}\right)Q_i+\left(\dfrac{\partial P_i}{\partial x_1}-\dfrac{\partial Q_i}{\partial x_2}\right)P_i, \quad	F_{i,2}=\left(\dfrac{\partial P_i}{\partial x_2}+\dfrac{\partial Q_i}{\partial x_1}\right)P_i+\left(\dfrac{\partial Q_i}{\partial x_2}-\dfrac{\partial P_i}{\partial x_1}\right)Q_i.\]
Let $g_i=(g_{i,1},g_{i,2})$ and observe that,
	\[(M_iX_i)\land\dfrac{\partial X_i}{\partial \mu_i}-\sin\theta_iR_{i,i} = k_i[g_{i,2}(\overline{P}_i-\sin\theta_iF_{i,2})-g_{i,1}(\overline{Q}_i+\sin\theta_iF_{i,1})].\]
Therefore, if we take $g_i=-\omega_0(-\overline{Q}_i-\sin\theta_iF_{i,1},\;\overline{P}_i-\sin\theta_iF_{i,2})$, then we can conclude that, 
\begin{equation}\label{18}
	d_i(\mu)=a_i\mu_i+O(||\mu||^2),
\end{equation}
with $a_i=\frac{\partial d_i}{\partial \mu_i}(0)>0$, $i\in\{1,\dots,n\}$. If $n=1$, then it follows from Proposition~\ref{Prop5} that any $\mu\in\mathbb{R}$ arbitrarily small such that $(R_1-1)\sigma_0\mu<0$ result in the bifurcation of at least one limit cycle. Suppose $n\geqslant2$ and that the result had been proved in the case $n-1$. We will now prove by induction in $n$. For definiteness we can assume $R_n>1$ and therefore $R_{n-1}<1$ and thus $r_n>1$. Moreover, it follows from Theorem~\ref{Main1} that $\Gamma^n$ is stable. Define,
	\[D=(d_1,\dots,d_{n-2},d_{n-1}^*).\]
It follows from Proposition~\ref{Prop3} and from \eqref{18} that we can apply the Implicit Function Theorem on $D$ and thus obtain unique $C^\infty$-functions $\mu_i=\mu_i(\mu_n)$, $\mu_i(0)=0$, $i\in\{1,\dots,n-1\}$, such that,
	\[D(\mu_1(\mu_n),\dots,\mu_{n-1}(\mu_n),\mu_n)=0,\]
for $|\mu_n|$ small enough. It also follows from \eqref{18} that $d_n\neq0$ if $\mu_n\neq0$, with $|\mu_n|$ small enough. Therefore, if $\mu_i=\mu_i(\mu_n)$ and $\mu_n\neq0$, then it follows from the definition of $D=0$ that there exist a $\Gamma^{n-1}=\Gamma^{n-1}(\mu_n)$ polycycle formed by $n-1$ singularities and $n-1$ regular orbits $L_i^*=L_i^*(\mu_n)$ such that,
\begin{enumerate}[label=\arabic*)]
	\item $\Gamma^{n-1}\to\Gamma^n$,
	\item $L_{n-1}^*\to L_n\cup L_{n-1}$ and,
	\item $L_i^*\to L_i$, $i\in\{1,\dots,n-2\}$,
\end{enumerate}
as $\mu_n\to0$. See Figure~\ref{FigExtra3}. 
\begin{figure}[ht]
	\begin{center}
		\begin{minipage}{6cm}
			\begin{center}
				\begin{overpic}[width=5.5cm]{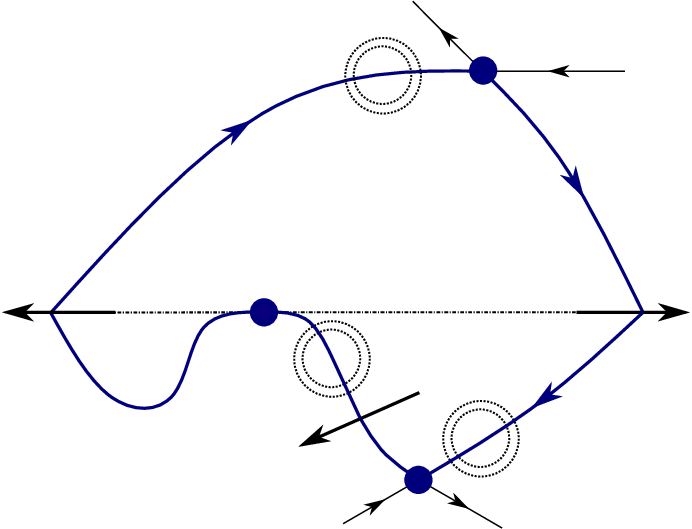} 
					\put(71,70){$p_1$}
					\put(35,25){$p_2$}
					\put(58,1){$p_3$}
					\put(75,32.5){$\Sigma$}
					\put(1,24){$l_1$}
					\put(46,6){$l_2$}
					\put(97,24){$l_3$}
				\end{overpic}
				
				Before the perturbation.
			\end{center}
		\end{minipage}
		\begin{minipage}{6cm}
			\begin{center}
				\begin{overpic}[width=5.5cm]{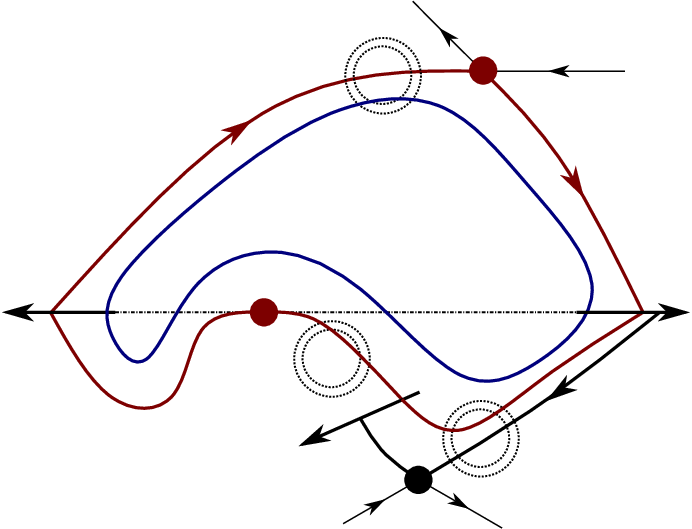}
					\put(71,70){$p_1$}
					\put(35,25){$p_2$}
					\put(58,1){$p_3$}
					\put(75,32.5){$\Sigma$}
					\put(1,24){$l_1$}
					\put(46,6){$l_2$}
					\put(97,24){$l_3$}
				\end{overpic}
				
				After the perturbation.
			\end{center}
		\end{minipage}
	\end{center}
	\caption{Illustration of the induction process with $R_3>1$ and $R_2<1$. Blue (resp. red) means a stable (resp. unstable) polycycle or limite cycle.}\label{FigExtra3}
\end{figure}
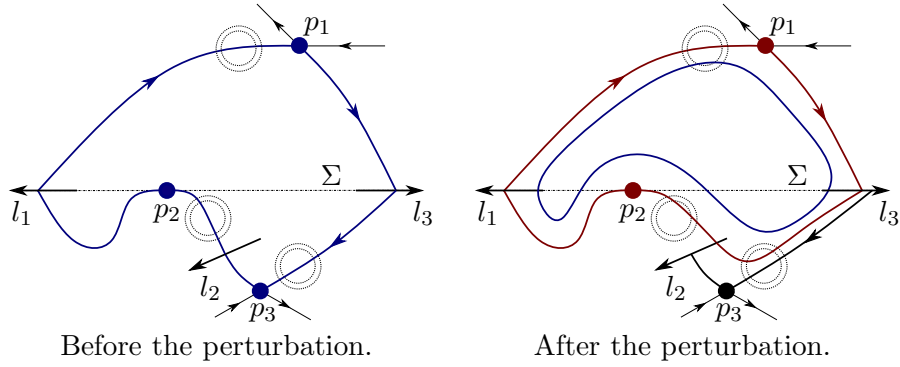
Let, 
	\[\left.R_j^*=\prod_{i=1}^{j}r_i\right|_{\mu_i=\mu_i(\mu_n),\;i\in\{1,\dots,n-1\}},\]
$j\in\{1,\dots,n-1\}$. Then it follows from the hypothesis, 
	\[(R_i-1)(R_{i+1}-1)<0,\]
for $i\in\{1,\dots,n-1\}$ and from the hypothesis $R_{n-1}<1$, that, 
	\[(R_i^*-1)(R_{i+1}^*-1)<0,\]
for $i\in\{1,\dots,n-2\}$ and $R_{n-1}^*<1$ for $\mu_n\neq0$ small enough. Thus, it follows from Theorem~\ref{Main1} that $\Gamma^{n-1}$ is unstable while $\Gamma^n$ is stable. It then follows from the Poincar\'e-Bendixson theory, and from the fact that the first return map is non-flat, that at least one stable limit cycle $\overline{\gamma}_n(\mu_n)$ exists near $\Gamma^{n-1}$. In fact both the limit cycle and $\Gamma^{n-1}$ bifurcates from $\Gamma^n$. Now fix $\mu_n\neq0$, $|\mu_n|$ arbitrarily small, and define the non-smooth system,
	\[Z_0^*=Z^*(x)+g^*(x,\overline{\mu}),\]
where $Z^*(x)=Z(x)+g(x,\mu_1(\mu_n),\dots,\mu_{n-1}(\mu_n),\mu_n)$ and, 
	\[g^*(x,\overline{\mu})=\sum_{i=1}^{n-1}\overline{\mu}_ik_i(x)g_i(x),\]
with $\overline{\mu}_i=\mu_i-\mu_i(\mu_n)$. It then follows by the definitions of $G_{i,j}$ and $L_i^*$ that,
	\[\Gamma^{n-1}\cap G_{i,j}=(L_i^*)^+\cap G_{i,j}\neq\emptyset,\]
$i\in\{1,\dots,n-1\}$ and $j\in\{1,2\}$. In this new parameter coordinate system the bump functions $k_i$ still ensures that $\frac{\partial d_i}{\partial \mu_j}(\mu)=0$ if $i\neq j$. Since $a_i>0$, it also follows that at the origin of this new coordinate system we still have $\frac{\partial d_i}{\partial \mu_i}(0)>0$. Therefore, it follows by induction that at least $n-1$ crossing limit cycles $\overline{\gamma}_j(\overline{\mu})$, $j\in\{1,\dots,n-1\}$, bifurcates near $\Gamma^{n-1}$ for arbitrarily small $|\overline{\mu}|$. Furthermore, we observe that $\overline{\gamma}_n(\mu_n)$ persists for $\overline{\mu}$ small enough, because it has odd multiplicity. {\hfill$\square$}

\section*{Acknowledgments}

We thank to the reviewers their comments and suggestions which help us to improve the presentation of this paper. The author is supported by S\~ao Paulo Research Foundation (FAPESP), grants 2019/10269-3 and 2021/01799-9.

\end{document}